\documentclass[a4paper,12pt]{article}
\usepackage{amsmath,amsthm,amsfonts,amssymb,color}
\usepackage[matrix,arrow]{xy}
\usepackage{graphicx}
\usepackage[UKenglish]{babel}
\usepackage{multicol} 
\usepackage{appendix}
\usepackage{epstopdf}

\usepackage{caption}
\usepackage{subcaption}

\usepackage{subeqnarray}

\usepackage{multirow} 
\usepackage{float}
\usepackage{graphicx}
\usepackage{longtable}  

\makeatletter
\renewcommand\@biblabel[1]{#1.}
\makeatother

\usepackage{amsmath,amssymb,amsthm}
\usepackage{enumerate}
\usepackage{amsfonts}
\usepackage{graphics,graphicx}
\usepackage{verbatim}
\usepackage{latexsym,makeidx}
\usepackage{amsmath,amscd}
\usepackage{appendix}

\newenvironment{keywords}{
  \vspace{2mm}
  \noindent
  \keywordsname: 
  \itshape\small
}

\def\keywordsname{\textbf{Keywords}}
  \def\mathsubclassname{\textbf{2010 AMS Subject Classification}}

\newtheorem{teo}{Theorem}[section]

\newcommand\be{\begin{equation}}
\newcommand\ee{\end{equation}}

\newcommand\pder[2][]{\ensuremath{\frac{\partial#1}{\partial#2}}}

\newcommand{\dee}{d}
\newcommand{\Ste}{\text{Ste}}
\newcommand{\Bt}{\text{Bi}}

 \newtheorem{thm}{Theorem}[section]
 \newtheorem{corollary}[thm]{Corollary}

 \theoremstyle{definition}
 
 \theoremstyle{remark}
 \newtheorem{remark}[thm]{Remark}
 
 \numberwithin{equation}{section}

\begin{document}

\title{APPROXIMATE SOLUTIONS TO ONE-PHASE STEFAN-LIKE PROBLEMS WITH SPACE-DEPENDENT LATENT HEAT}

\author{
Julieta Bollati$^{1}$,  Domingo A. Tarzia $^{1}$\\ \\
\small {{$^1$} Depto. Matem\'atica - CONICET, FCE, Univ. Austral, Paraguay 1950} \\  
\small {S2000FZF Rosario, Argentina.}\\
\small{Email: JBollati@austral.edu.ar; DTarzia@austral.edu.ar.} 
}
\date{}

\maketitle

\begin{abstract}
The work in this paper concerns the study of different approximations for   one-dimensional one-phase Stefan-like problems with a space-dependent latent heat. It is considered two different problems, which differ  from each other in  their boundary condition imposed at the fixed face: Dirichlet and Robin conditions.
The approximate solutions are obtained by applying the heat balance integral method (HBIM), a modified heat balance integral method, the refined integral method (RIM) .  Taking advantage of the exact analytical solutions we compare and test the accuracy of the approximate solutions. The analysis is carried out using the  dimensionless generalized Stefan number (Ste) and Biot number (Bi). It is also studied the case when Bi goes to infinity in the problem with a convective condition, recovering the approximate  solutions  when a temperature condition is imposed at the fixed face. Some numerical simulations are provided in order to  assert which of the approximate integral methods turns out to be optimal. Moreover, we pose an approximate technique based on minimizing the least-squares error, obtaining also approximate solutions for the classical Stefan problem.\\ \\
\end{abstract}



\maketitle
\section{Introduction}

Stefan problems model  heat transfer processes that involve a change of phase. They constitute a broad field of study since they appear in  a great number of mathematical and industrial significance problems 
 \cite{AlSo1993}, \cite{Ca1984}, \cite{Gu2003}, \cite{Lu1991}. A large bibliography on the subject is given in \cite{Ta2000} and a review on analytical  solutions in \cite{Ta2011}.

The Stefan problem with a space-dependent latent heat can be  found in several physical processes.  In \cite{SwVoPa2000}, it was developed a mathematical model for the shoreline movement in a sedimentary basin using an analogy with the one-phase melting Stefan problem with a variable latent heat. Besides, in  \cite{ZhShZh2018}, it was introduced a two-phase Stefan problem with a general type of space-dependent latent heat from the background of the artificial ground-freezing technique.

The assumption of  variable latent heat not only becomes meaningful in the study of the shoreline movement or in the soil freezing techniques but also in the nanoparticle melting \cite{RiMy2016} and in the one-dimensional consolidation with threshold gradient \cite{ZhBuLu2013}. More references dealing with non-constant latent heat can be found in: \cite{VoSwPa2004}, \cite{ZhWaBu2014},\cite{ZhXi2015}, \cite{SaTa2011-a}, \cite{BoTa2018-CAA}, \cite{BoTa2018-ZAMP}, \cite{Do2014}, \cite{Mc1991}, \cite{ZhHuLiZhZh2018}, \cite{Pr1970}, \cite{Go2002}.

  In this paper we are going to consider two different  Stefan-like problems (P) and (P$_h$) with space-dependent latent heat imposing different conditions at the fixed boundary.
  The first problem to consider can be stated as follows:
  
\noindent\textbf{Problem (P)}. Find the location of the free boundary $x=s(t)$ and the temperature $T=T(x,t)$ at the liquid region $0<x<s(t)$ such that:\\[-0.25cm]
\begin{subequations}\label{Pinfty-A}
\be
 \pder[T]{t}=a^2 \pder[^2T]{x^2},\qquad  0<x<s(t), \quad t>0, \label{EcCalor:1F-pos-tempinfty-A}
\ee
\be 
  T(0,t)=\theta_{_\infty} t^{\alpha/2},\qquad  t>0, \label{FrontFija:1F-pos-tempinfty-A} 
\ee
\be 
  T(s(t),t)=0,\qquad  t>0, \label{TempFront:1F-pos-tempinfty-A} 
\ee
\be 
  k\pder[T]{x}(s(t),t)=-\gamma s(t)^{\alpha} \dot s(t), \qquad t>0, \label{CondStefan:1F-pos-tempinfty-A}
\ee
\be 
  s(0)=0,\label{FrontInicial:1F-pos-tempinfty-A} 
\ee
\end{subequations}

The equation  (\ref{EcCalor:1F-pos-tempinfty-A})  is the heat conduction equation in the liquid region where $a^2=\frac{k}{\rho c}$  is the diffusion coefficient being $k$ the thermal conductivity, $\rho$ the density mass and $c$ the specific heat capacity.  At  $x=0$, a Dirichlet condition (\ref{FrontFija:1F-pos-tempinfty-A}) is  imposed. It must be noticed that the temperature at the fixed boundary is time-dependent and it is characterized by a parameter $\theta_{_\infty}>0$. In addition, condition (\ref{TempFront:1F-pos-tempinfty-A}) represents the fact that the phase change temperature  is assumed to be 0 without loss of generality, condition (\ref{CondStefan:1F-pos-tempinfty-A}) is the corresponding Stefan condition and (\ref{FrontInicial:1F-pos-tempinfty-A}) is the initial position of the free boundary.

The remarkable feature of the problem is related to the condition at
the interface given by the Stefan condition (\ref{CondStefan:1F-pos-tempinfty-A}), where the latent heat by unit of volume is space-dependent defined by a power function of the position $\tfrac{\gamma}{\rho} x^{\alpha}(t)$ with $\gamma$ a given positive constant and  $\alpha$ an arbitrarily non-negative real value.

The second problem (P$_h$) arises by imposing a convective (Robin) condition at the fixed face $x=0$ instead of a Dirichlet one.  In mathematical terms, we can define (P$_h$) as:

\noindent \textbf{Problem (P$_h$)}. Find the  location of the free boundary $x=s_{_{h}}(t)$ and the temperature  $T_{h}=T_{h}(x,t)$  at the liquid region $0<x<s_{h}(t)$   such that equations (\ref{EcCalor:1F-pos-tempinfty-A}), (\ref{TempFront:1F-pos-tempinfty-A})-(\ref{FrontInicial:1F-pos-tempinfty-A}) are satisfied, together with the Robin condition
\begin{equation}\label{FrontFijaConvectiva}
k\pder[T]{x}(0,t)=\frac{h}{\sqrt{t}}\left[T(0,t)-\theta_{_\infty}t^{\alpha/2} \right],\qquad t>0. \tag{\ref{FrontFija:1F-pos-tempinfty-A}$^\star$}
\end{equation}  
Condition (\ref{FrontFijaConvectiva}) states that the incoming heat flux at the fixed face is proportional to the difference between the material temperature and the ambient temperature.
Here, $\theta_{_\infty}t^{\alpha/2}$ characterizes the bulk temperature at a large distance from the fixed face $x = 0$ and $h$ represents the heat transfer at the fixed face. We will work under the assumption that $h>0$
and \mbox{$0<T_{{h}}(0,t)<\theta_{_\infty}t^{\alpha/2}$} in order to guarantee the melting process.

 The exact solution to   problem (P) was given  in \cite{ZhWaBu2014} for integer non-negative values of $\alpha$ and was generalized in \cite{ZhXi2015}  by taking $\alpha$ as a real non negative constant. Besides, the exact solution of the problem (P$_h$) was provided in \cite{BoTa2018-CAA}.

It is known that due to the non-linear nature of the Stefan problem, exact solutions are limited to a few cases and therefore it is necessary to solve them either numerically or approximately. 

The idea in this paper is to take advantage of the exact solutions available in the literature  testing the accuracy of different approximate integral  methods.

The heat balance integral method, introduced by Goodman \cite{Go1958}, is an approximate technique which is usually employed for solving the location of the free boundary in phase-change problems. It  consists in the transformation of  the heat equation into an ordinary differential equation in time, assuming a quadratic profile in space for the temperature. For those profiles, several variants have been introduced in \cite{Wo2001} and  \cite{SaSiCo2006}. In addition,  in \cite{Hr2009-a}, \cite{Hr2009-b}, \cite{Mi2012}, \cite{MiMy2010-a}  this method has been applied defining new accurate temperature profiles. Moreover, for the case $\alpha=0$, the explicit solution to the problem (P$_{{h}}$) for the two-phase process was given in \cite{Ta2017} and this was useful to obtain the accuracy of different heat balance integral methods to problem (P$_{{h}}$) in \cite{BoSeTa2018}.


The paper will be structured as follows: in Section 2 we will give a briefly introduction about the approximate methods to be implemented. Then, in Section 3, we will recall the exact solution to  problem (P) that considers a Dirichlet condition at the fixed face and we will get some different approximate solutions that will be tested with the exact one. In Section 4, we will present the exact solution to the problem with a Robin condition at the fixed face, i.e. problem (P$_{{h}}$). We are going to implement the different approximate methods and we will test their accuracy. In all cases, we are going to provide numerical examples and comparisons. In addition, we will show that the approximate solutions to problem  (P$_{{h}}$) converge to the approximate solutions to problem (P) when the heat transfer coefficient $h$ goes to infinity.
Finally, in Section 5, we will implement an approximate method that consists in minimizing the least-squares error as in \cite{RiMyMc2019}. For the case $\alpha=0$ we obtain different approximations for the problems (P) and (P$_{{h}}$) by using the least-squares approximate method.

\section{Heat balance integral methods}

The classical heat balance integral method, described for first time in \cite{Go1958}, was designed to approximate problems involving phase-changes. This method consists in changing the heat equation (\ref{EcCalor:1F-pos-tempinfty-A}) by an ordinary differential equation in time that arises by: assuming a suitable temperature profile consistent with the boundary conditions, integrating (\ref{EcCalor:1F-pos-tempinfty-A})  with respect to the spacial variable in an appropiate interval, and replacing the Stefan condition (\ref{CondStefan:1F-pos-tempinfty-A}) by a new equation obtained from the phase-change temperature (\ref{TempFront:1F-pos-tempinfty-A}).

Therefore, if we derive condition (\ref{TempFront:1F-pos-tempinfty-A}) with respect to time,  and take into account the heat equation (\ref{EcCalor:1F-pos-tempinfty-A}) we get
\be
\pder[T]{x}(s(t),t) \dot s(t)+a^2 \pder[^2 T]{x^2} (s(t),t)=0. 
\ee
Clearing $\dot s$ and replacing it in the Stefan condition (\ref{CondStefan:1F-pos-tempinfty-A}) it gives
\be
\frac{k}{\gamma s^{\alpha}(t)}\left[ \pder[T]{x}(s(t),t) \right]^2= a^2 \pder[^2 T]{x^2} (s(t),t). \tag{\ref{CondStefan:1F-pos-tempinfty-A}$^\star$}
\ee
This last condition is going to substitute the Stefan condition in the approximated problem obtained from the classical heat balance integral method.

On the other hand, using equation  (\ref{EcCalor:1F-pos-tempinfty-A}) and the condition (\ref{TempFront:1F-pos-tempinfty-A}) we have
\begin{subeqnarray*}
\dfrac{\dee}{\dee t} \int\limits_{0}^{s(t)} T(x,t) \dee x&=&\int\limits_{0}^{s(t)} \pder[T]{t}(x,t)\dee x +T(s(t),t)\dot{s}(t) \\ 
&=& \int\limits_{0}^{s(t)} a^2 \pder[^2 T]{x^2}(x,t)\dee x=a^2\left[\pder[T]{x}(s(t),t)-\pder[T]{x}(0,t) \right].
\end{subeqnarray*}
Then, by applying the Stefan condition  (\ref{CondStefan:1F-pos-tempinfty-A}) it results that
\be 
\frac{\dee}{\dee t} \int\limits_{0}^{s(t)} T(x,t) \dee x= -a^2\left[\frac{\gamma}{k}s^{\alpha} (t) \dot s(t)+\pder[T]{x}(0,t) \right].\tag{\ref{EcCalor:1F-pos-tempinfty-A}$^\star$}
\ee

The {\bf  classical heat balance integral method}, approximate problem (P) through  a new problem that arises from replacing the heat equation (\ref{EcCalor:1F-pos-tempinfty-A}) by (\ref{EcCalor:1F-pos-tempinfty-A}$^\star$) and the Stefan condition (\ref{CondStefan:1F-pos-tempinfty-A}) by  (\ref{CondStefan:1F-pos-tempinfty-A}$^\star$) keeping the rest of the conditions of (P) the same. In short,  the method consists in solving the  problem goberned by (\ref{EcCalor:1F-pos-tempinfty-A}$^\star$), (\ref{FrontFija:1F-pos-tempinfty-A}),(\ref{TempFront:1F-pos-tempinfty-A}), (\ref{CondStefan:1F-pos-tempinfty-A}$^\star$) and (\ref{FrontInicial:1F-pos-tempinfty-A}). A priori, this method will work better than the classical one due to the fact that it changes less conditions from the exact problem.

In \cite{Wo2001}, a {\bf modified integral balance method} is presented. It postulates to change only the heat equation keeping the same the rest of conditions, even the Stefan condition. It means that it consists in solving an approximate problem given by (\ref{EcCalor:1F-pos-tempinfty-A}$^\star$), (\ref{FrontFija:1F-pos-tempinfty-A}), (\ref{TempFront:1F-pos-tempinfty-A}), (\ref{CondStefan:1F-pos-tempinfty-A}) and (\ref{FrontInicial:1F-pos-tempinfty-A}).

On the other hand, from the heat equation  (\ref{EcCalor:1F-pos-tempinfty-A}), and the condition (\ref{TempFront:1F-pos-tempinfty-A}) we have
\begin{subeqnarray*}
\int\limits_0^{s(t)} \int\limits_0^x \pder[T]{t}(z,t) \dee z \dee x &=& \int\limits_0^{s(t)}  \int\limits_0^x a^2 \pder[^2T]{z^2}(z,t) \dee z\; \dee x \\
&=& \int\limits_0^{s(t)} a^2 \left[\pder[T]{x}(x,t)-\pder[T]{x}(0,t) \right] \dee x  \\
&=& a^2 \left[T(s(t),t)-T(0,t)-\pder[T]{x}(0,t)s(t) \right],
\end{subeqnarray*}
that is to say
\be
\int\limits_0^{s(t)} \int\limits_0^x \pder[T]{t}(z,t) \dee z \dee x =  -a^2 \left[T(0,t)+\pder[T]{x}(0,t)s(t) \right].\tag{\ref{EcCalor:1F-pos-tempinfty-A}$^{\dag}$}
\ee

The  {\bf refined integral method} introduced in \cite{SaSiCo2006} suggests to solve an approximate problem given by (\ref{EcCalor:1F-pos-tempinfty-A}$^\dag$), (\ref{FrontFija:1F-pos-tempinfty-A}), (\ref{TempFront:1F-pos-tempinfty-A}), (\ref{CondStefan:1F-pos-tempinfty-A}) and (\ref{FrontInicial:1F-pos-tempinfty-A}). That is to say, to replace the heat equation (\ref{EcCalor:1F-pos-tempinfty-A}) by (\ref{EcCalor:1F-pos-tempinfty-A}$^{\dag}$).

In all cases, to solve the above approximated problems, it is necessary to adopt a suitable profile for the temperature. Throughout  this paper we will assume a quadratic profile in space
\begin{equation}\label{Perfil}
\widetilde{T}(x,t)=t^{\alpha/2}\theta_{_\infty}\left[ \widetilde{A}\left( 1-\frac{x}{\widetilde{s}(t)}\right)+\widetilde{B}\left(  1-\frac{x}{\widetilde{s}(t)}\right)^2\right], 
\end{equation}

where $\widetilde{T}$ and $\widetilde{s}$ will be approximations of $T$ y $s$ respectively. We can notice that in the chosen  profile a power function of time arises in order to be compatible with the boundary conditions imposed in the exact problem.

It is worth to mention that for the approximations to the problem (P$_{h}$), it will be enough to consider the same approximate problems stated for (P),  changing only the boundary condition (\ref{FrontFija:1F-pos-tempinfty-A}) by (\ref{FrontFija:1F-pos-tempinfty-A}$^\star$).

\section{One-phase Stefan problem with Dirichlet condition}

\subsection{Exact solution}

\medskip

Before introducing the different approaching methods for problem (P),  we  present the exact solution, which was given in \cite{ZhWaBu2014} and \cite{ZhXi2015} for the cases when $\alpha\in\mathbb{N}_0$ and $\alpha\in\mathbb{R}^+\setminus\mathbb{N}_0$, respectively.

Let us define the following non-dimensional parameter 
\be
\text{Ste}=\frac{k\theta_{_\infty}}{\gamma a^{\alpha+2}}\label{Ste}
\ee
which is called generalized Stefan number. 
We use the word ``generalized''  since in case that the latent heat $l$ is constant, i.e. $\alpha=0$, we can  recover the usual formula for the Stefan number, which 
assuming a zero phase-change temperature is given by $\text{Ste}=\frac{c\theta_\infty}{l}$. Notice that if we take $\alpha=0$ then the Dirichlet condition at the fixed face is given by $\theta_\infty$ and from the Stefan condition (\ref{CondStefan:1F-pos-tempinfty-A})  the latent heat becomes $l=\gamma/\rho$.

Then, if we combine the results found  in    \cite{ZhWaBu2014} and \cite{ZhXi2015} we can rewrite the solution of the problem (P) (as it was done in the appendix of \cite{BoTa2018-EJDE}), obtaining for each $\alpha\in\mathbb{R}^+_0$ that: 
\begin{align}
&T(x,t)=t^{\alpha/2}\left[AM\left(-\frac{\alpha}{2},\frac{1}{2},-\eta^2 \right) +B\eta M\left(-\frac{\alpha}{2}+\frac{1}{2},\frac{3}{2},-\eta^2 \right)\right],\\
&s(t)=2a\nu \sqrt{t},
\end{align}
where $\eta=\frac{x}{2a\sqrt{t}}$ is the similarity variable, 
\be A=\theta_{_\infty},\qquad \qquad B=\frac{-\theta_{_\infty}M\left( -\frac{\alpha}{2},\frac{1}{2},-{\nu^2}\right)}{\nu M\left(-\frac{\alpha}{2}+\frac{1}{2},\frac{3}{2},-{\nu^2} \right)},
\ee
and $\nu$ is the unique positive solution to the following equation
\be\label{EcNuA}
 \frac{\text{Ste}}{2^{\alpha+1}}f(z)=z^{\alpha+1}, \qquad z>0,
\ee
where  is defined by
\be 
f(z)=\frac{1}{zM\left(\frac{\alpha}{2}+1,\frac{3}{2},z^2 \right)} \label{f}
\ee
 and $M(a,b,z)$ is the Kummer function defined by
\begin{align}
& M(a,b,z)=\sum\limits_{s=0}^{\infty}\frac{(a)_s}{(b)_s s!}z^s ,\qquad\qquad \text{ (b cannot be a  nonpositive integer)} \label{M} 
\end{align}
being $(a)_s$  the Pochhammer symbol:
\begin{equation}
 (a)_s=a(a+1)(a+2)\dots (a+s-1), \quad \quad (a)_0=1 
\end{equation}

\begin{remark}\label{ExactaMenor1}
If $0<\text{Ste}<1$, the unique solution $\nu$ of equation (\ref{EcNuA})  belongs to the interval $ (0,1)$. In fact,  define $H(x)=\tfrac{\text{Ste}}{2^{\alpha+1}}f(z)-z^{\alpha+1}$. On one hand we have 
$H(0)=+\infty$ due to the fact that $M\left(\tfrac{\alpha}{2}+1,\tfrac{3}{2},0\right)=1$. On the other hand, we obtain $H(1)<0$ as $\tfrac{\text{Ste}}{2^{\alpha+1}}<1< M\left(\tfrac{\alpha}{2}+1,\tfrac{3}{2},1 \right)$.
\end{remark}

\subsection{Approximate solutions}

\medskip

We are going to implement the different approximate techniques for the problem (P) and test their accuracy taking advantage of the knowledge of the exact solution.

First of all, we introduce a problem (P$_{1}$) which arises when applying the classical heat balance integral problem to (P). According to the previous section, the \textbf{problem (P$_{1}$)}  consists in finding the free boundary $s_{1}=s_{1}(t)$  and the temperature $T_{1}=T_{1}(x,t)$ in $0<x<s_{1}(t)$ such that conditions (\ref{EcCalor:1F-pos-tempinfty-A}$^\star$), (\ref{FrontFija:1F-pos-tempinfty-A}),(\ref{TempFront:1F-pos-tempinfty-A}), (\ref{CondStefan:1F-pos-tempinfty-A}$^\star$) and (\ref{FrontInicial:1F-pos-tempinfty-A}) are verified.

Provided that $T_{1}$  assumes a quadratic profile in space like (\ref{Perfil}) we get the following result
\begin{teo}\label{TeoP1}
If $0<\text{Ste}<1$, there exists at least one solution to problem  $\mathrm{(P_{1})}$, given by
\begin{eqnarray}
T_{1}(x,t)&=&t^{\alpha/2}\theta_{_\infty} \left[ A_{1}\left(1-\frac{x}{s_{1}(t)}\right) +B_{1} \left(1-\frac{x}{s_{1}(t)}\right)^{2}\right], \label{PerfilT1}\\
s_{1}(t)&=& 2a\nu_{1} \sqrt{t},
\end{eqnarray}
where the constants $A_{1}, B_{1}$ are defined as a function of $\nu_{1}$ by:
\begin{align} 
A_{1}&=\frac{-2\left[ 3\; 2^\alpha \nu_{1}^{\alpha+2}+\mathrm{Ste}\left((-3+(1+\alpha)\nu_{1}^2\right)\right]}{\mathrm{Ste} \left(3+(1+\alpha)\nu_{1}^2\right)},\qquad\label{A1} \\
B_{1}&= \frac{3\left[2^{\alpha+1}\nu_{1}^{\alpha+2}+\mathrm{Ste}\left(-1+(1+\alpha)\nu_{1}^2\right) \right]}{\mathrm{Ste} \left(3+(1+\alpha)\nu_{1}^2\right)},\label{B1}
\end{align}
and the coefficient $\nu_{1}$ is a solution to the following equation
\begin{align}
& z^{2\alpha+4}(-3)\; 2^{2\alpha+1} (\alpha-2)+z^{2\alpha+2}(-9)\; 2^{2\alpha+1} +z^{4+\alpha} \left(-3\right)\; 2^\alpha (\alpha-3)(\alpha+1)\mathrm{Ste}\nonumber\\
&+z^{\alpha+2}\left( -3\right) \; 2^{\alpha+1} (\alpha+7)\mathrm{Ste}+z^{\alpha} 9\; 2^{\alpha} \mathrm{Ste}+z^4 2 (\alpha+1)^2\mathrm{Ste}^2\nonumber\\
&+z^2 (-12)(\alpha+1)\mathrm{Ste}^2+18\mathrm{Ste}^2=0, \qquad z>0.\label{EcNu1}
\end{align}
\end{teo}

\begin{proof}
First of all we shall notice that if $T_{1}$ adopts the profile (\ref{PerfilT1}), it is clear evident that the condition  (\ref{TempFront:1F-pos-tempinfty-A}) is automatically verified.  From the imposed Dirichlet condition at the fixed boundary  (\ref{FrontFija:1F-pos-tempinfty-A}) we get
\be \label{1}
A_{1}+B_{1}=1
\ee
In addition, we have that
$$\pder[T_{1}]{x}(x,t)=-t^{\alpha/2} \theta_{_\infty}\left[\frac{A_{1}}{s_{1}(t)}+\frac{2B_{1}}{s_{1}(t)}\left(1-\frac{x}{s_1(t)}\right) \right],$$
and
$$
\pder[^2T_{1}]{x^2}(x,t)=t^{\alpha/2}\theta_{_\infty} \frac{2B_{1}}{s_1^2(t)}.
$$
Therefore, from condition (\ref{CondStefan:1F-pos-tempinfty-A}$^\star$) we claim
$$\frac{k}{\gamma s^{\alpha}_{1}(t)} t^{\alpha} \theta_{_\infty}^2 \frac{A_{1}^2}{s_1^2 (t)}=a^2 t^{\alpha/2} \theta_{_\infty}\frac{2B_{1}}{s_1^2(t)}. $$
Then, it follows that
$$s_{1}(t)= \left( \frac{A_{1}^2}{2 B_{1}} \frac{k \theta_{_\infty}}{\gamma a^2}\right)^{1/\alpha} \sqrt{t}.$$
Defining $\nu_{1}$ such that $\nu_{1}=\frac{1}{2a}\left( \frac{A_{1}^2}{2 B_{1}} \frac{k \theta_{_\infty}}{\gamma a^2}\right)^{1/\alpha}$, we deduce that
\be \label{s1Demo}
s_{1}(t)=2a\nu_{1}\sqrt{t}
\ee
where $\nu_{1}$, $A_{1}$ and $B_{1}$  are related as
\be\label{2}
A_{1}^2 =\frac{2^{\alpha+1} \nu_{1}^{\alpha}}{\Ste} B_{1}.
\ee
Condition  (\ref{EcCalor:1F-pos-tempinfty-A}$^\star$) and
\begin{align*}
\frac{\dee}{\dee t} \int\limits_{0}^{s_1(t)} T_1(x,t) \dee x& =\frac{\dee}{\dee t} \int\limits_{0}^{s_1(t)} t^{\alpha/2} \theta_{_\infty} \left[ A_{1}\left(1-\frac{x}{s_{1}(t)}\right) +B_{1}\left(1-\frac{x}{s_{1}(t)}\right)^{2}\right] \dee x  \\
&= \theta_{_\infty}\left(\frac{A_{1}}{2}+\frac{B_{1}}{3} \right) \left(\frac{\alpha}{2}t^{\alpha/2-1}s_{1}(t)+t^{\alpha/2}\dot s_{1}(t) \right),
\end{align*}
gives 
\be 
 \theta_{_\infty}\left(\tfrac{A_{1}}{2}+\tfrac{B_{1}}{3} \right) \left(\tfrac{\alpha}{2}t^{\alpha/2-1}s_{1}(t)+t^{\alpha/2}\dot s_{1}(t) \right)= -a^2 \left[\tfrac{\gamma}{k}s_{1}^{\alpha} (t) \dot s_{1}(t)+t^{\alpha/2}\theta_{_\infty}\tfrac{(A_{1}+2B_{1})}{s_{1}(t)} \right].
\ee
According to (\ref{s1Demo}), it results that
\be \label{3}
A_{1}\left( (\alpha+1)\nu_{1}^2-1 \right) +B_{1} \left( \tfrac{2}{3} (\alpha+1)\nu_{1}^2-2\right)=\tfrac{-2^{\alpha+1}\nu_{1}^{\alpha+2}}{\Ste}.
\ee

Thus, we have obtained three equations (\ref{1}), (\ref{2}) and (\ref{3}) for the unknown coefficients $A_{1}$, $B_{1}$ and $\nu_{1}$.

From (\ref{1}) and (\ref{3}) it is obtained that $A_{1}$ and $B_{1}$ are given as a function of $\nu_{1}$ by (\ref{A1}) and (\ref{B1}), respectively.

Then, equation (\ref{2}) leads to the fact that  $\nu_{1}$ must be  a positive solution to (\ref{EcNu1}).

For the existence of solution to problem (P$_{1}$) it remains to prove that the function $w_{1}=w_{1}(z)$, defined as the left hand side of equation (\ref{EcNu1}), has at least one positive root. This can be easily check by evaluating $w_{1}(0)=18\Ste^2>0$ and
$$w_{1}(1)=-\alpha^2 (3\; 2^{\alpha}-2\Ste)\Ste-2\alpha(3 \; 4^{\alpha}+4\Ste^2)-2(3\;  4^{\alpha}+3\; 2^{\alpha+2}\Ste-4\Ste^2)$$
From the assumption that $0<\Ste<1$, we obtain $3\; 2^{\alpha}-2\Ste>0$, and $$3\;  4^{\alpha}+3\; 2^{\alpha+2}\Ste-4\Ste^2>2^{\alpha+2}\Ste-4\Ste^2=4\Ste(3\; 2^{\alpha}-\Ste)>0.$$
Therefore $w_{1}(1)<0$. Consequently,  we can assure that there exists at least one positive solution to equation (\ref{EcNu1}) in the interval $(0,1)$.
\end{proof}

\begin{remark} The approximated free boundary  $s_{1}$ behaves as  a square root of time just like the exact one $s$, it means that 
$s_{1}(t)= 2a\nu_{1} \sqrt{t}$ while $s(t)= 2a\nu \sqrt{t}$.
\end{remark}

\begin{remark}
After Theorem \ref{TeoP1} follows the question about uniqueness of solution.
We found that there exists different values for $\alpha$ and $0<\text{Ste}<1$ that leads to multiple roots of equation (\ref{EcNu1}), i.e. $w_1(z)=0,\; z>0$ (see Figure \ref{Fig:Nu1NoUnico})

\begin{figure}[h!!]
\begin{center}

\includegraphics[scale=0.2]{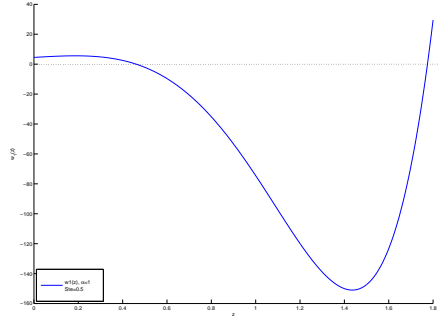}

\end{center}
\caption{{\footnotesize  Plot of $w_1(z)$   for  $\alpha=1$ and $\text{Ste}=0.5$ }}\label{Fig:Nu1NoUnico}

\end{figure}

However our study must be reduced to find  the roots of  $w_1(z)$ located in the interval $(0,1)$ in view of the proof of Theorem \ref{TeoP1} but also in view of Remark \ref{ExactaMenor1}. For the particular case of $\alpha=0$ the uniqueness analysis was given in \cite{BoSeTa2018}.

Although we could not prove it analytically, by setting different values for $\alpha$ and $\Ste$ we can see that there exists just one root of the polynomial $w_1(z)$ located in the interval $(0,1)$. In Figure \ref{Fig:Nu1-Ste05} we illustrate this fact setting $\alpha=0.5, 1,1.5,2,3,5,10$ and $\Ste=0.5$. We have just plot between $0\leq z\leq 0.5$ in order to appreciate better this fact.

\begin{center}
\begin{figure}[h!!]
\begin{center}

\includegraphics[scale=0.3]{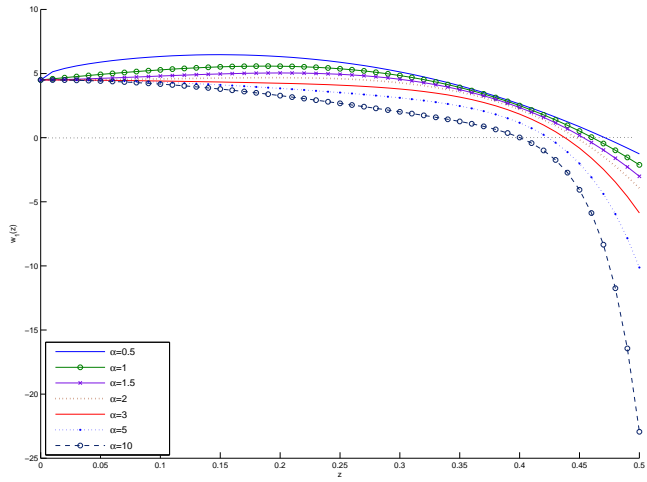}
\end{center}

\caption{{\footnotesize  Plot of $w_1(z)$   for different values of  $\alpha$ setting $\text{Ste}=0.5$ }}\label{Fig:Nu1-Ste05}

\end{figure}
\end{center}

\end{remark}

\medskip

With the purpose of testing the classical integral balance method and in view of the above remark  we will only compare graphically  the coefficient $\nu_{1}$ that characterizes the approximated free boundary problem $s_{1}$ with the coefficient $\nu$ that characterizes the exact free boundary $s$. In Figure  \ref{Fig:Nu1VsNu},  we illustrate this comparisons for different values of $0<\Ste<1$ and  $\alpha$.

\begin{figure}[h!!]
\begin{center}

\includegraphics[scale=0.4]{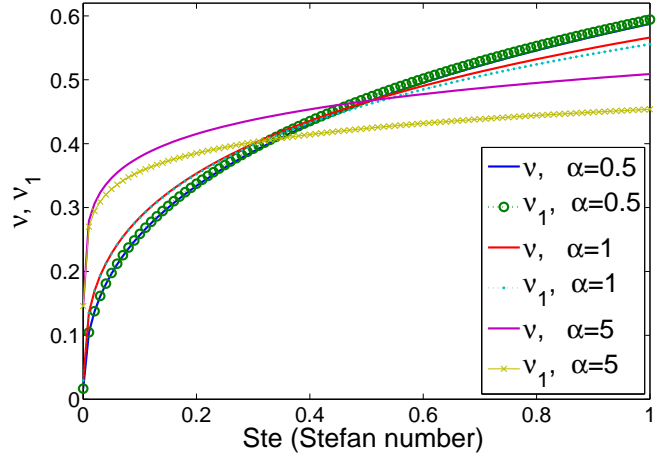}
\end{center}

\begin{center}
\caption{{\footnotesize  Plot of $\nu$ and $\nu_{_1}$ against $\Ste$ for different values of $\alpha=0.5,1,5$ }}\label{Fig:Nu1VsNu}

\end{center}
\end{figure}

For the comparisons we have assumed that $0<\Ste<1$ not only due to the hypothesis in Theorem \ref{TeoP1}, but also because of the fact that in general,  the majority of phase change materials under a realistic temperature present a Stefan number that does not exceed 1 (see \cite{So1979}).

\vspace{1cm}


Now, we will turn to the modified integral balance method. In this case we state an approximated \textbf{problem  (P$_{2}$)} for the problem (P) that is stated as follows: find the free boundary $s_{2}=s_{2}(t)$ and the temperature $T_{2}=T_{2}(x,t)$  in $0<x<s_{2}(t)$ such that equation (\ref{EcCalor:1F-pos-tempinfty-A}$^\star$) and conditions (\ref{FrontFija:1F-pos-tempinfty-A}), (\ref{TempFront:1F-pos-tempinfty-A}), (\ref{CondStefan:1F-pos-tempinfty-A}) and (\ref{FrontInicial:1F-pos-tempinfty-A}) are satisfied.

Assuming a quadratic profile in space for $T_{2}$  we obtain the next theorem

\begin{teo}\label{TeoP2}
The problem $\mathrm{(P_{2})}$ has a unique solution given by:
\begin{eqnarray}
T_{2}(x,t)&=&t^{\alpha/2}\theta_{_\infty} \left[ A_{2}\left(1-\frac{x}{s_{2}(t)}\right) +B_{2} \left(1-\frac{x}{s_{2}(t)}\right)^{2}\right],\label{PerfilT2} \\
s_{2}(t)&=& 2a\nu_{2} \sqrt{t},
\end{eqnarray}
where the constantes $A_{2}$ and $B_{2}$ are given by
\begin{align} 
A_{2}&=\frac{6\mathrm{Ste}-2\mathrm{Ste}\;\nu_{2}^2 (\alpha+1)-3\; 2^{\alpha+1} \nu_{2}^{\alpha+2}}{\mathrm{Ste }\;\left(\nu_{2}^2 (\alpha+1)+3 \right)},\label{A2}\\
B_{2}&=\frac{-3\mathrm{Ste}\;+3\mathrm{Ste}\;\nu_{2}^2(\alpha+1)+3\; 2^{\alpha+1} \nu_{2}^{\alpha+2}}{\mathrm{Ste }\;\left(\nu_{2}^2 (\alpha+1)+3 \right)},\label{B2}
 \end{align}
and where $\nu_2$ is the unique positive solution to the equation
\be \label{EcNu2}
z^{\alpha+4} 2^{\alpha} (\alpha+1)+z^{\alpha+2}3\; 2^{\alpha+1}+z^2 \mathrm{Ste} (\alpha+1)-3\mathrm{Ste}=0, \qquad z>0.
\ee
\end{teo}

\begin{proof}
Condition (\ref{TempFront:1F-pos-tempinfty-A}) is clearly checked from the chosen temperature profile.

From the Stefan condition (\ref{CondStefan:1F-pos-tempinfty-A})  we obtain
\be
-kt^{\alpha/2}\theta_{_\infty}\tfrac{A_{2}}{s_{2}(t)}=-\gamma s_{2}^{\alpha}(t)\dot s_{2}(t).
\ee
Therefore it results that
\be
s_{2}(t)=\left( \frac{(\alpha+2)}{(\frac{\alpha}{2}+1)} \frac{k\theta_{_\infty}}{\gamma}A_{2}\right)^{1/(\alpha+2)} \sqrt{t}.
\ee
If we introduce the coefficient $\nu_{2}$ such that $ \nu_{2}=\frac{1}{2a}\left( \frac{(\alpha+2)}{(\frac{\alpha}{2}+1)} \frac{k\theta_{_\infty}}{\gamma}A_{2}\right)^{1/(\alpha+2)}$, the free boundary can be expressed as
\be 
s_{2}(t)=2a\; \nu_{2}\sqrt{t},
\ee
where the following relation holds
\be\label{RelP2-10}
A_{2}=\frac{2^{\alpha+1} \nu_{2}^{\alpha+2}}{\Ste} .
\ee
Taking into account the boundary condition at the fixed face (\ref{FrontFija:1F-pos-tempinfty-A}) we get
\be \label{RelP2-20}
A_{2}+B_{2}=1.
\ee
In addition, in virtue of the equation (\ref{EcCalor:1F-pos-tempinfty-A}$^\star$) we get
\be \label{RelP2-30}
A_{2}\left( (\alpha+1)\nu_{2}^2-1 \right) +B_{2} \left( \tfrac{2}{3} (\alpha+1)\nu_{2}^2-2\right)=\tfrac{-2^{\alpha+1}\nu_{2}^{\alpha+2}}{\Ste}.
\ee
From equations (\ref{RelP2-10}), (\ref{RelP2-20}) and (\ref{RelP2-30}) we claim that $A_{2}$ and $B_{2}$ can be written in 
function of  $\nu_{2}$ through formulas (\ref{A2}) and (\ref{B2}), respectively. In addition, $\nu_{2}$ must be a solution to the equation  (\ref{EcNu2}). 
So that, to finish the proof, it remains to show that equation (\ref{EcNu2}) has a unique positive solution, i.e the function defined by the left hand side of this equation $w_{2}=w_{2}(z)$ has a unique positive root. This is easily checked by noting that
$$w_{2}(0)=-3\Ste<0,\qquad w_{2}(+\infty)=+\infty, \qquad \frac{\dee w_{2}}{\dee z}(z)>0,\quad \forall z>0.$$ 

\end{proof}

In the Figure \ref{Fig:Nu2VsNu}, as we did for the classical heat balance integral method, we compare the coefficients $\nu_{2}$ (approximate)  with $\nu$ (exact) for different values of $0<\Ste<1$ and $\alpha$.

\begin{figure}[h!!]
\begin{center}
\includegraphics[scale=0.4]{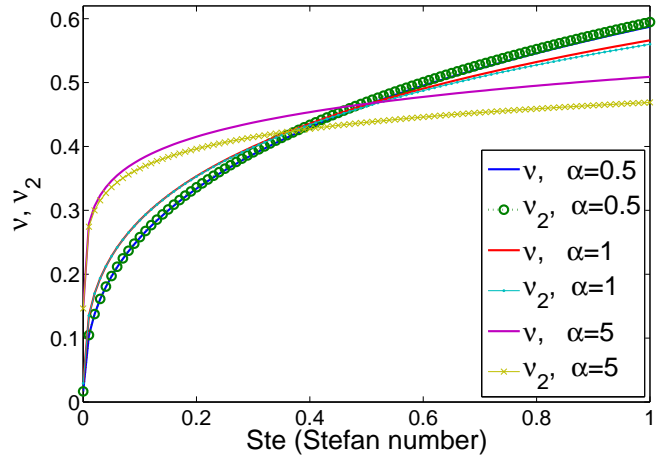}
\end{center}
\caption{{\footnotesize  Plot of $\nu$ and $\nu_{2}$ against $\Ste$ for different values of $\alpha=0.5,1,5$ }}\label{Fig:Nu2VsNu}
\end{figure}

The refined integral method intends to approximate the problem (P) through solving a\textbf{ problem (P$_{3}$)} that consists in finding 
the free boundary $s_{3}=s_{3}(t)$ and the temperature $T_{3}=T_{3}(x,t)$  in $0<x<s_{3}(t)$ such that equation (\ref{EcCalor:1F-pos-tempinfty-A}$^\dag$) and conditions (\ref{FrontFija:1F-pos-tempinfty-A}), (\ref{TempFront:1F-pos-tempinfty-A}), (\ref{CondStefan:1F-pos-tempinfty-A}) and (\ref{FrontInicial:1F-pos-tempinfty-A}) are satisfied.

Under the assumption that $T_{3}$ adopts a quadratic profile in space like (\ref{Perfil}) we can state the following result:

\begin{teo}\label{TeoP3}
The unique solution to problem $\mathrm{(P_{3})}$ is given by
\begin{eqnarray}
T_{3}(x,t)&=&t^{\alpha/2}\left[ A_{3}\theta_{_\infty}\left(1-\frac{x}{s_{3}(t)}\right) +B_{3}\theta_{_\infty} \left(1-\frac{x}{s_{3}(t)}\right)^{2}\right],\label{PerfilT3} \\
s_{3}(t)&=& 2a\nu_{3} \sqrt{t},
\end{eqnarray}
where the constants $A_{3}$ and  $B_{3}$ are given by
\begin{align} 
A_{3}&=\frac{6\mathrm{Ste}-2\mathrm{Ste}\;\nu_{3}^2 (\alpha+1)-3\; 2^{\alpha+1} \nu_{3}^{\alpha+2}}{\mathrm{Ste }\;\left(\nu_{3}^2 (\alpha+1)+3 \right)},\label{A3}\\
B_{3}&=\frac{-3\mathrm{Ste}\;+3\mathrm{Ste}\;\nu_{3}^2(\alpha+1)+3\; 2^{\alpha+1} \nu_{3}^{\alpha+2}}{\mathrm{Ste }\;\left(\nu_{3}^2 (\alpha+1)+3 \right)},\label{B3}
 \end{align}
and where $\nu_{3}$  is the unique solution to equation
\be \label{EcNu3}
z^{\alpha+4} 2^{\alpha+1} \alpha+z^{\alpha+2}3\; 2^{\alpha+2}+z^2 \mathrm{Ste} (2+3\alpha)-6\mathrm{Ste}=0, \qquad z>0.
\ee
\end{teo}

\begin{proof}
The proof is similar to the one of the Theorem \ref{TeoP2}. The only difference to take into account is the fact that equation(\ref{EcCalor:1F-pos-tempinfty-A}$^{\dag}$) is equivalent to
\be
\nu_{3}^2 \left[A_{3}\left( \frac{1}{3}+\frac{2}{3}\alpha\right)+B_{3}\left( \frac{1}{3}+\frac{\alpha}{2}\right) \right]=B_{3} 
\ee

\end{proof}

\medskip
In the Figure \ref{Fig:Nu3VsNu}  we compare graphically the coefficient $\nu_{3}$ that characterizes the approximate free boundary $s_{3}$ with the coefficient  $\nu$ that characterizes the exact boundary $s$.

\begin{figure}[h!!]
\begin{center}
\includegraphics[scale=0.4]{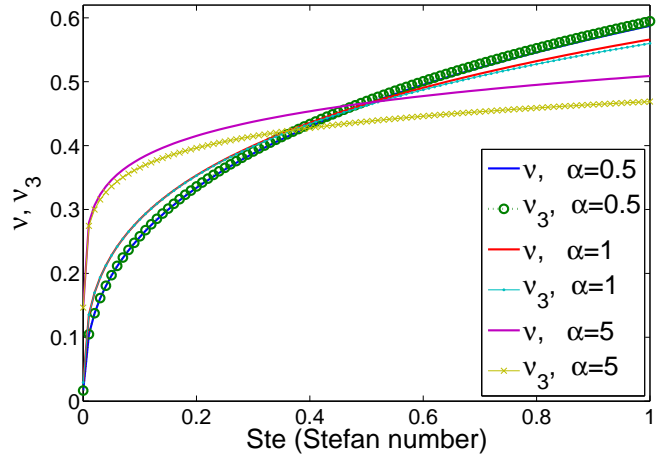}
\end{center}
\caption{{\footnotesize  Plot of $\nu$ and $\nu_{3}$ against $\Ste$ for different values of $\alpha=0.5,1,5$ }}\label{Fig:Nu3VsNu}
\end{figure}

\subsection{Comparisons between the approximate solutions and the exact one}

\medskip
 
In the previous section we have  applied three different methods to approximate the solution to the Stefan problem  (P), with a Dirichlet condition at the fixed face and a variable latent heat.

For each method  we have stated a problem (P$_{{i}}$), ${i}=1,2,3$ and we have compared graphically the dimensionless coefficients $\nu_{{i}}$ that characterizes their free boundaries  $s_{{i}}$, with the coefficient $\nu$  that characterizes the exact free boundary $s$.

Then the goal will be to compare numerically, for different Stefan number, the coefficient $\nu$ given by (\ref{EcNuA}) with the approximate coefficients $\nu_{1}$, $\nu_{2}$  and $\nu_{3}$ defined by (\ref{EcNu1}), (\ref{EcNu2}) and (\ref{EcNu3}), respectively.

In order that the comparisons be more representative, in Tables \ref{Tabla:NuiVsNu1}-\ref{Tabla:NuiVsNu3}  we show the exact values obtained for $\nu$, the approximate value $\nu_{{i}}$ and percentage error committed  in each case $E(\nu_{{i}})=100\left\vert\frac{\nu-\nu_{{i}}}{\nu} \right\vert$, ${i}=1,2,3$ for different values of $\Ste$ and $\alpha$.

\begin{table}
\small
\caption{{\footnotesize Dimensionless coefficients of the free boundaries and their percentage relative error for $\alpha=0$.}}
\label{Tabla:NuiVsNu1}  
\begin{center}
\begin{tabular}{cc|cc|cc|cc}
\hline
Ste      & $\nu$     &   $\nu_1$ & $E_{\text{rel}}(\nu_1)$       &   $ \nu_{2} $  & $E_{\text{rel}}(\nu_{2})$      &   $\nu_{3}$ & $E_{\text{rel}}(\nu_{3})$   \\
\hline
	0.1  &  0.2200  &  0.2232 &   1.4530 \% &   0.2209 &   0.3947 \%  &   0.2218  &  0.7954 \% \\
    0.2  &  0.3064  &  0.3143 &   2.5729 \% &   0.3087 &   0.7499 \%  &   0.3111  &  1.5213 \% \\
    0.3  &  0.3699  &  0.3827 &   3.4575 \% &   0.3738 &   1.0707 \%  &   0.3780  &  2.1856 \% \\
    0.4  &  0.4212  &  0.4388 &   4.1687 \% &   0.4270 &   1.3618 \%  &   0.4330  &  2.7953 \% \\
    0.5  &  0.4648  &  0.4869 &   4.7478 \% &   0.4723 &   1.6266 \%  &   0.4804  &  3.3561 \% \\
    0.6  &  0.5028  &  0.5290 &   5.2236 \% &   0.5122 &   1.8683 \%  &   0.5222  &  3.8729 \% \\
    0.7  &  0.5365  &  0.5666 &   5.6173 \% &   0.5477 &   2.0895 \%  &   0.5599  &  4.3501 \% \\
    0.8  &  0.5669  &  0.6006 &   5.9443 \% &   0.5799 &   2.2923 \%  &   0.5941  &  4.7913 \% \\
    0.9  &  0.5946  &  0.6316 &   6.2165 \% &   0.6094 &   2.4786 \%  &   0.6255  &  5.1999 \% \\
    1.0  &  0.6201  &  0.6600 &   6.4432 \% &   0.6365 &   2.6500 \%  &   0.6547  &  5.5786 \% \\
    \hline
\end{tabular}
\end{center}
\end{table}

\begin{table}
\small
\caption{{\footnotesize Dimensionless coefficients of the free boundaries and their percentage relative error for $\alpha=0.5$.}}
\label{Tabla:NuiVsNu2}  
\begin{center}
\begin{tabular}{cc|cc|cc|cc}
\hline
Ste      & $\nu$     &   $\nu_{_1}$ & $E_{\text{rel}}(\nu_{_1})$       &   $ \nu_{_2} $  & $E_{\text{rel}}(\nu_{_2})$      &   $\nu_{_3}$ & $E_{\text{rel}}(\nu_{_3})$   \\
\hline

    0.1 &   0.2569 &   0.2587  &  0.6956  \% &  0.2574 &   0.2001 \% &   0.2580 &   0.4012 \% \\
    0.2 &   0.3339 &   0.3372  &  0.9999  \% &  0.3349 &   0.3147 \% &   0.3360 &   0.6321 \% \\
    0.3 &   0.3876 &   0.3921  &  1.1718  \% &  0.3891 &   0.3974 \% &   0.3907 &   0.7995 \% \\
    0.4 &   0.4298 &   0.4353  &  1.2678  \% &  0.4318 &   0.4596 \% &   0.4338 &   0.9260 \% \\
    0.5 &   0.4650 &   0.4711  &  1.3143  \% &  0.4674 &   0.5067 \% &   0.4698 &   1.0225 \% \\
    0.6 &   0.4953 &   0.5018  &  1.3264  \% &  0.4980 &   0.5423 \% &   0.5007 &   1.0959 \% \\
    0.7 &   0.5220 &   0.5288  &  1.3133  \% &  0.5249 &   0.5684 \% &   0.5280 &   1.1508 \% \\
    0.8 &   0.5458 &   0.5528  &  1.2814  \% &  0.5491 &   0.5869 \% &   0.5523 &   1.1905 \% \\
    0.9 &   0.5675 &   0.5745  &  1.2352  \% &  0.5709 &   0.5989 \% &   0.5744 &   1.2173 \% \\
    1.0 &   0.5873 &   0.5943  &  1.1777  \% &  0.5909 &   0.6054 \% &   0.5946 &   1.2334 \% \\

    \hline
\end{tabular}
\end{center}
\end{table}

\begin{table}\small
\caption{{\footnotesize Dimensionless coefficients of the free boundaries and their percentage relative error for $\alpha=5$.}}
\label{Tabla:NuiVsNu3}  
\begin{center}
\begin{tabular}{cc|cc|cc|cc}
\hline
Ste      & $\nu$     &   $\nu_{_1}$ & $E_{\text{rel}}(\nu_{_1})$       &   $ \nu_{_2} $  & $E_{\text{rel}}(\nu_{_2})$      &   $\nu_{_3}$ & $E_{\text{rel}}(\nu_{_3})$   \\
\hline
	0.1  &  0.3793  &  0.3563 &   6.0700 \% &    0.3723 &   1.8469 \%  &   0.3656  &  3.6135 \% \\
    0.2  &  0.4151  &  0.3849 &   7.2853 \% &    0.4055 &   2.3333 \%  &   0.3963  &  4.5496 \% \\
    0.3  &  0.4374  &  0.4020 &   8.0816 \% &    0.4256 &   2.6810 \%  &   0.4145  &  5.2154 \% \\
    0.4  &  0.4537  &  0.4143 &   8.6859 \% &    0.4403 &   2.9615 \%  &   0.4276  &  5.7505 \% \\
    0.5  &  0.4667  &  0.4239 &   9.1776 \% &    0.4518 &   3.2010 \%  &   0.4377  &  6.2058 \% \\
    0.6  &  0.4775  &  0.4317 &   9.5943 \% &    0.4612 &   3.4122 \%  &   0.4460  &  6.6060 \% \\
    0.7  &  0.4869  &  0.4384 &   9.9572 \% &    0.4693 &   3.6025 \%  &   0.4529  &  6.9656 \% \\
    0.8  &  0.4950  &  0.4442 &  10.2795 \% &    0.4763 &   3.7766 \%  &   0.4589  &  7.2936 \% \\
    0.9  &  0.5023  &  0.4492 &  10.5699 \% &    0.4826 &   3.9376 \%  &   0.4642  &  7.5962 \% \\
    1.0  &  0.5090  &  0.4538 &  10.8345 \% &    0.4881 &   4.0880 \%  &   0.4689  &  7.8780 \% \\

    \hline
\end{tabular}
\end{center}
\end{table}

\newpage

From the tables, we can notice that for  $\alpha=0.5$, the error committed by each method is lowen than for $\alpha=0$ or $\alpha=5$. In all cases, the method which shows the greatest accuracy is the modified integral balance method. In other words, the best approximate problem to (P) is given by  problem (P$_{_2}$).

Besides, we can also provide an illustration ot the exact temperature $T$  with the approximate temperatures $T_{_{i}}$, ${i}=1,2,3$, given by (\ref{PerfilT1}), (\ref{PerfilT2}) and (\ref{PerfilT3}), respectively. If we consider $\alpha=5$, $\Ste=0.5$, $\theta_{_\infty}=30$ and  $a=1$ we obtain Figures (4)-(7)

\begin{figure}[h!]\centering
   \begin{minipage}{0.49\textwidth}
   \begin{center}
    \includegraphics[scale=0.3]{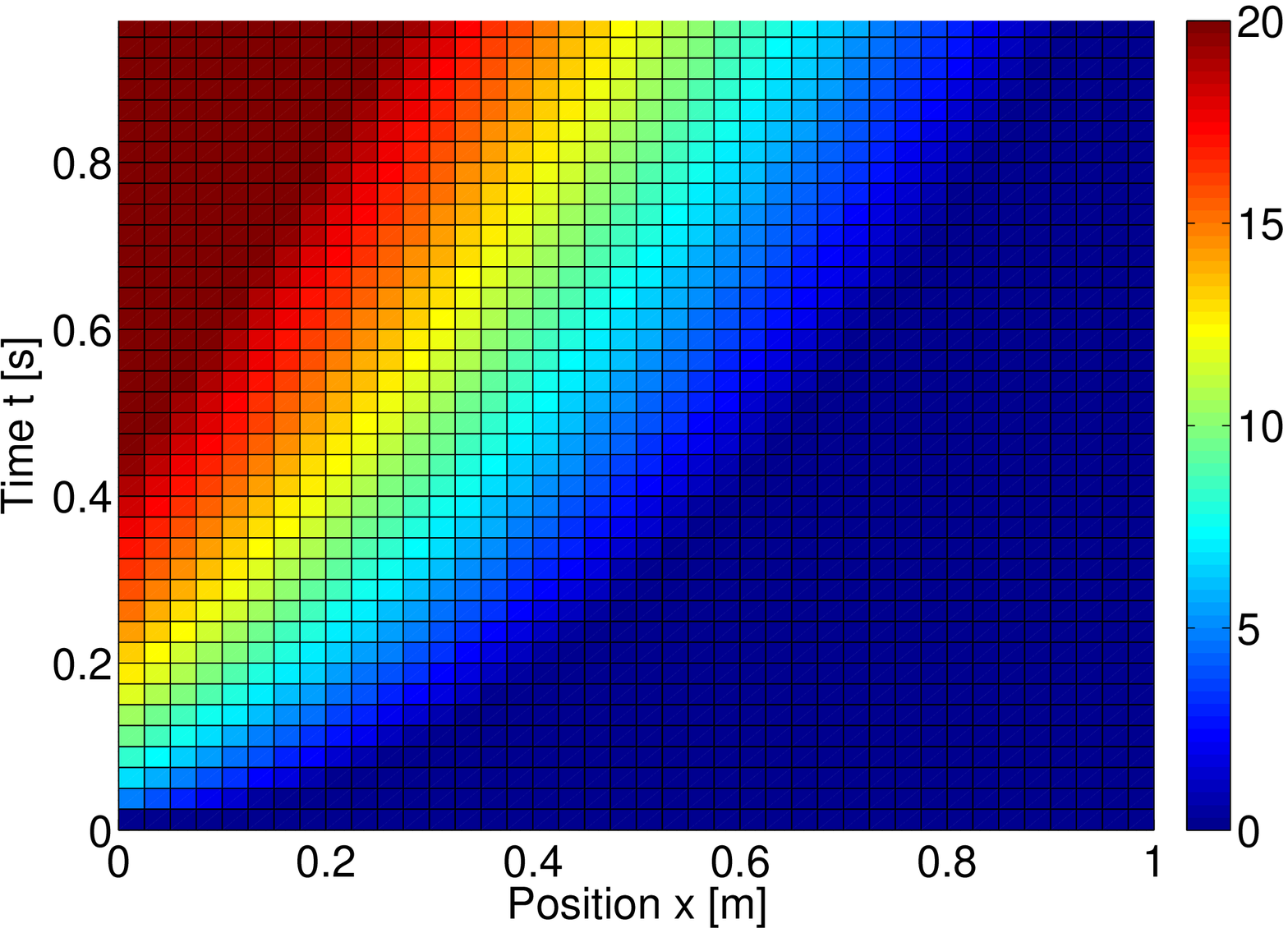}
     \caption{{\footnotesize Colour map for $T$ }}
   \end{center}
\label{Fig:Exacta}
   \end{minipage}
   \begin {minipage}{0.49\textwidth}
\begin{center}
	\includegraphics[scale=0.3]{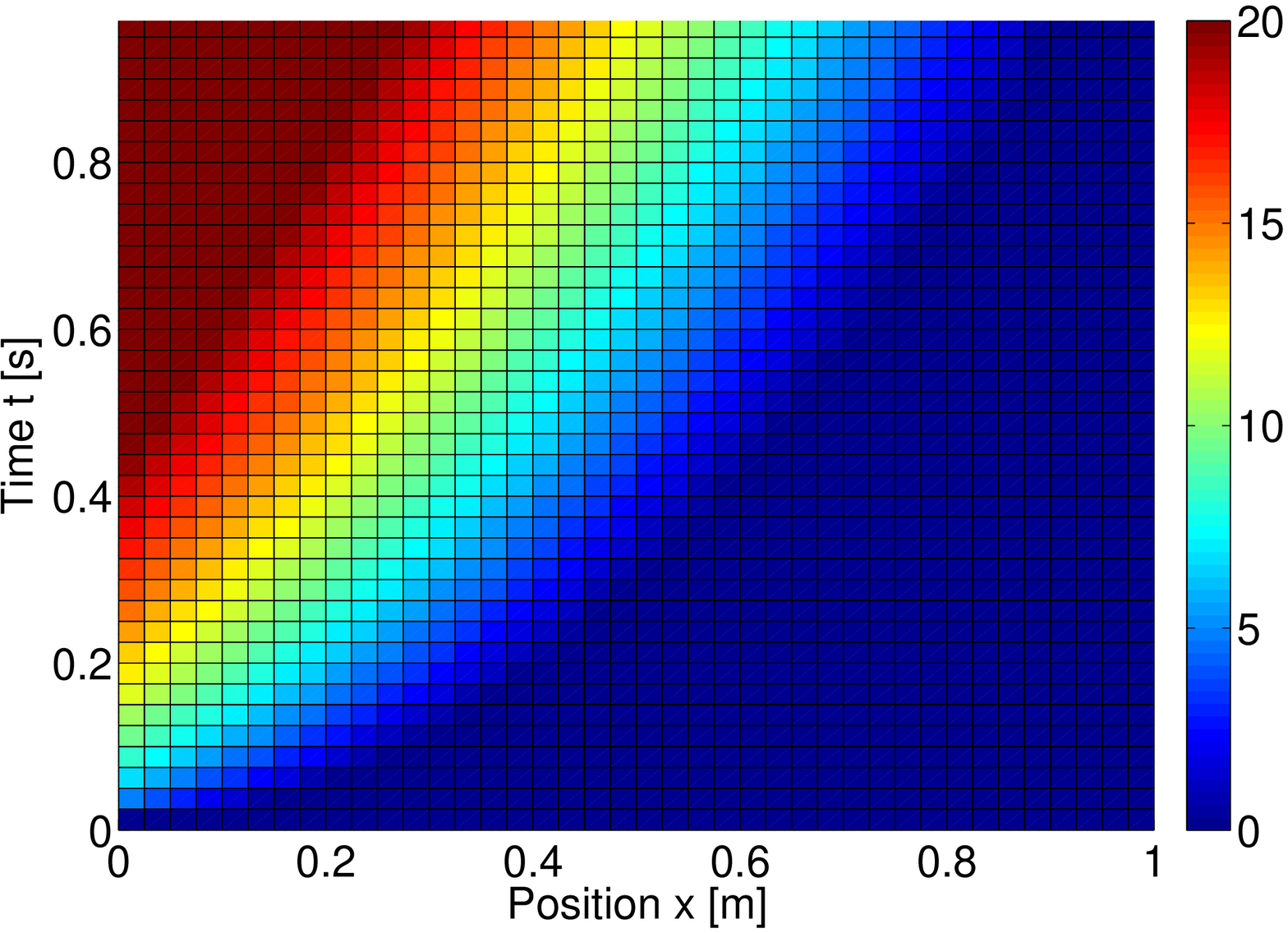}
	  \caption{ {\footnotesize  Colour map for  $T_{_1}$}}
\end{center}
   \label{Fig:BIC}
   \end{minipage}
\end{figure}

\begin{figure}[h!]\centering
   \begin{minipage}{0.49\textwidth}
   \begin{center}
    \includegraphics[scale=0.3]{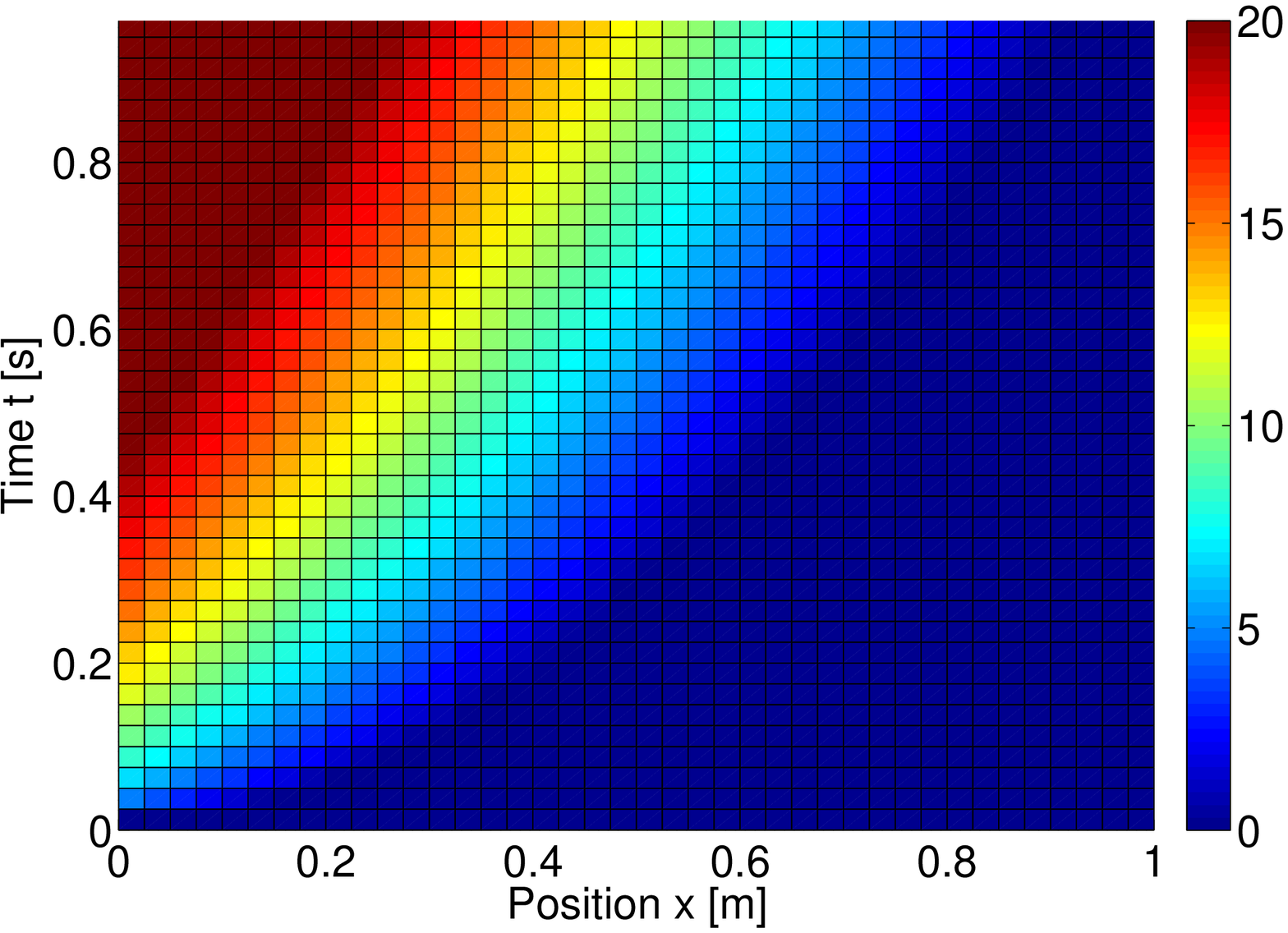}
     \caption{{\footnotesize Colour map for $T_{_2}$ }}
   \end{center}
\label{Fig:BIM}
   \end{minipage}
   \begin {minipage}{0.49\textwidth}
\begin{center}
	\includegraphics[scale=0.3]{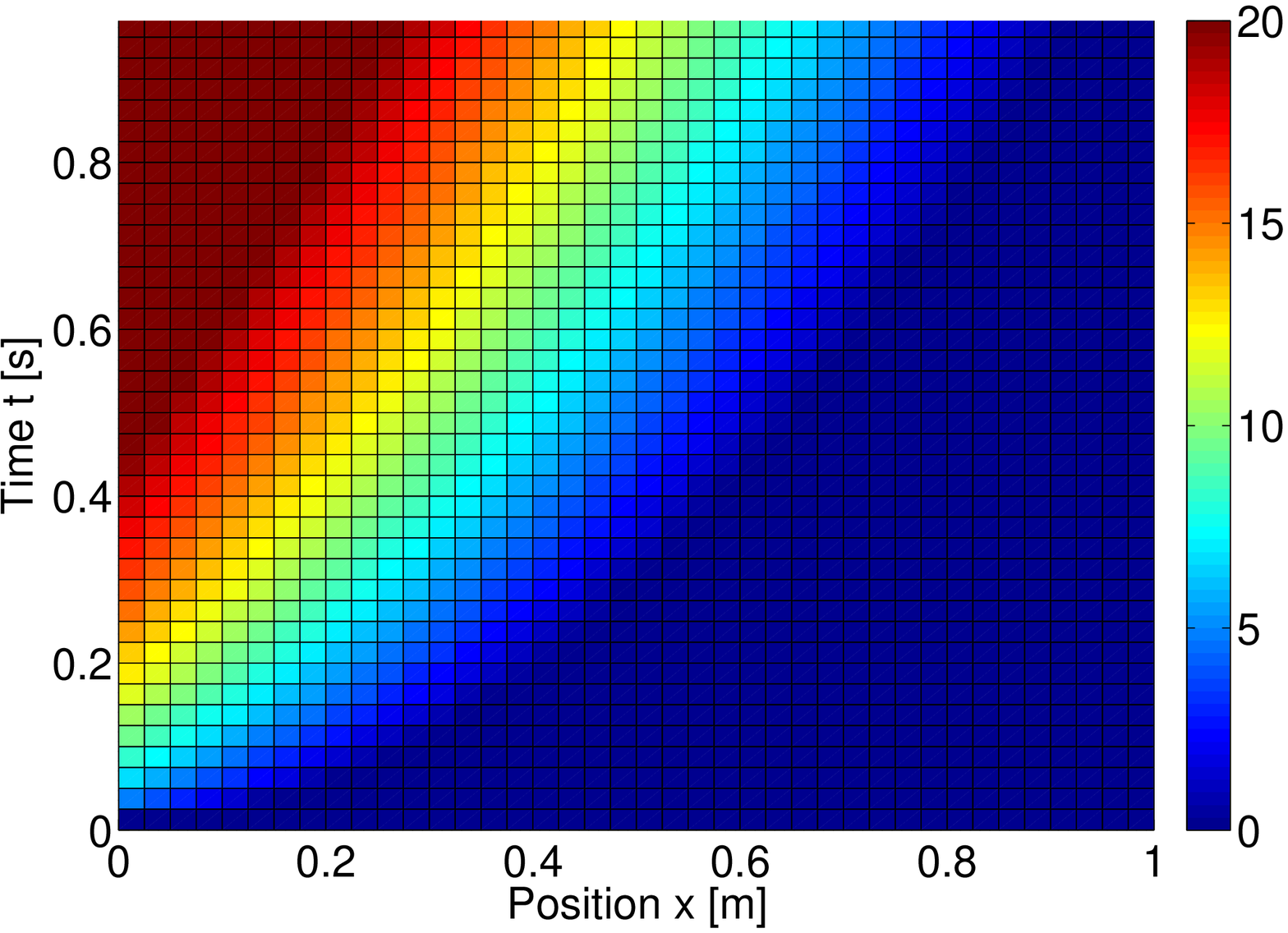}
	  \caption{ {\footnotesize  Colour map for $T_{_3}$}}
\end{center}
   \label{Fig:RIM}
   \end{minipage}
\end{figure}

\section{One-phase Stefan problem with Robin condition}
In this section we are going to present the exact solution of the problem with a Robin condition, then we will obtain different approximate solutions that will be compared and we will analyse their convergence when the coefficient that characterizes the heat transfer at the fixed boundary goes to infinity.

\subsection{Exact solution}

\medskip

We recall that the exact solution to problem (P$_{h}$) governed by equations (\ref{EcCalor:1F-pos-tempinfty-A}), (\ref{FrontFijaConvectiva}), (\ref{TempFront:1F-pos-tempinfty-A})-(\ref{FrontInicial:1F-pos-tempinfty-A}) given in \cite{BoTa2018-CAA} can be written as
\begin{align}
&T_{h}(x,t)=  t^{\alpha/2}\left[ A_{h} M\left(-\frac{\alpha}{2}, \frac{1}{2},-\eta^2\right)+B_{h} \eta M\left(-\frac{\alpha}{2}+\frac{1}{2},\frac{3}{2},-\eta^2\right)\right], \\
&s_{h}(t) =2 a \nu_{h}\sqrt{ t}, 
\end{align}
where $\eta=\frac{x}{2a\sqrt{t}}$ is the similarity variable, the coefficients $A_{h}$ and $B_{h}$ are given by
\begin{align}
& A_{h}=\frac{-\nu_{h} M\left(-\frac{\alpha}{2}+\frac{1}{2},\frac{3}{2},-\nu_{h}^2\right)}{M\left(-\frac{\alpha}{2},\frac{1}{2},-\nu_{h}^2\right)}B_{h}, \\
& B_{h}=\frac{- \theta_{_\infty} M\left( -\frac{\alpha}{2},\frac{1}{2},-\nu_{h}^2\right)}{\left[\frac{1}{2\Bt} M\left( -\frac{\alpha}{2},\frac{1}{2},-\nu_{h}^2\right)+ \nu_{h} M\left( -\frac{\alpha}{2}+\frac{1}{2},\frac{3}{2},-\nu_{h}^2\right) \right]}, 
\end{align}
and with $\nu_{h}$ defined as the unique solution to the following equation
\be \label{EcNuh}
\frac{\Ste}{ 2^{\alpha+1} } \frac{1}{\left[\frac{1}{f(z)}+\frac{1}{2\Bt } M\left(\frac{\alpha}{2}+\frac{1}{2},\frac{1}{2},z^2 \right)\right]} =z^{\alpha+1},\qquad z>0,
\ee
where  $\Ste$ and $f$ are  given by (\ref{Ste}) and  (\ref{f}), respectively and where the Biot number is defined by $\Bt=\frac{ah}{k}$. 

In \cite{BoTa2018-CAA} it was also proved that the unique solution to the exact problem with convective condition (P$_{h}$) converges pointwise to the unique solution to the problem with temperature condition (P) when the Biot number goes to infinity (i.e $h\to \infty$)

\subsection{Approximate solutions and convergence}

\medskip

As it was done for the problem (P), we will now apply the classical integral balance method, the modified integral balance method and the refined integral method to the problem (P$_{h}$). For each method we will going to state an approximate problem (P$_{_{{ih}}}$), $i=1,2,3$. Assuming a quadratic profile in space we are going to obtain the solutions to the approximate problems. Finally, we will show that the solution of each problem (P$_{_{{ih}}}$) converges to the solution of the problem (P$_{i}$) defined in the previous section, when $h\to\infty$. This fact is intuitively expected because the same happens to the exact problems (P$_{h}$) and (P).

We introduce an approximate \textbf{problem (P$_{_{{1h}}}$)} that arises when applying the classical heat balance integral method to the problem (P$_{h}$). It consists in finding the free boundary $s_{_{{1h}}}=s_{_{{1h}}}(t)$ and the temperature $T_{_{{1h}}}=T_{_{{1h}}}(x,t)$  in $0<x<s_{_{{1h}}}$ such that conditions: (\ref{EcCalor:1F-pos-tempinfty-A}$^\star$), (\ref{FrontFijaConvectiva}),(\ref{TempFront:1F-pos-tempinfty-A}), (\ref{CondStefan:1F-pos-tempinfty-A}$^\star$) and (\ref{FrontInicial:1F-pos-tempinfty-A}) are satisfied.

Provided that $T_{_{1h}}$ adopts a quadratic profile in space, like (\ref{Perfil}) we can prove the next result:

\begin{teo}
If $0<\mathrm{Ste}<1$, $\alpha\geq 0$ and $\mathrm{Bi}$ is large enough, there exists at least one solution to problem  $\mathrm{(P_{_{1h}})}$, which is given by
\begin{eqnarray}
T_{_{1h}}(x,t)&=&t^{\alpha/2}\theta_{_\infty} \left[ A_{_{1h}}\left(1-\frac{x}{s_{_{1h}}(t)}\right) +B_{_{1h}} \left(1-\frac{x}{s_{_{1h}}(t)}\right)^{2}\right],\label{PerfilT1h} \\
s_{_{1h}}(t)&=& 2a\nu_{_{1h}} \sqrt{t},
\end{eqnarray}
where the constants $A_{_{1h}}$ and $B_{_{1h}}$ are defined as a function of $\nu_{_{1h}}$
\begin{align} 
A_{_{1h}}&=\frac{6\mathrm{Ste}-2\mathrm{Ste}\; \nu^2_{_{1h}}(\alpha+1)-\frac{3}{\mathrm{Bi}} 2^{\alpha+1} \nu_{_{1h}}^{\alpha+1}-3\; 2^{\alpha+1}\nu_{_{1h}}^{\alpha+2} }{\mathrm{Ste}\left[\nu_{_{1h}}^2 (\alpha+1)+\frac{2}{\mathrm{Bi}} \nu_{_{1h}}(\alpha+1)+3 \right]},\qquad\label{A1h} \\[0.2cm]
B_{_{1h}}&= \frac{-3\mathrm{Ste}+3\mathrm{Ste}\; \nu_{_{1h}}^2(\alpha+1)+\frac{3}{\mathrm{Bi}} 2^{\alpha} \nu_{_{1h}}^{\alpha+1}+3\;2^{\alpha+1} \nu_{_{1h}}^{\alpha+2} }{\mathrm{Ste}\left[\nu_{_{1h}}^2 (\alpha+1)+\frac{2}{\mathrm{Bi}} \nu_{_{1h}}(\alpha+1)+3 \right]},\label{B1h}
\end{align}
where $\nu_{_{1h}}$ is a solution to the following equation
\begin{align}
&z^{2\alpha+4} (-3)2^{2\alpha+1}(\alpha-2)+ z^{2\alpha+3}(-3)\frac{2^{2\alpha}}{\mathrm{Bi}}(5\alpha-7)+z^{2\alpha+2}(-3)2^{2\alpha+1}\left(\frac{\alpha-2}{\mathrm{Bi}^2} +3\right)\nonumber \\
&+z^{2\alpha+1}(-9)\frac{2^{2\alpha}}{\mathrm{Bi}}+z^{\alpha+4} (-3)2^{\alpha}\mathrm{Ste} (\alpha-3)(\alpha+1)+z^{\alpha+3}(-3)\frac{2^{\alpha+1}}{\mathrm{Bi}} \mathrm{Ste} (\alpha-1) (\alpha+1)\nonumber \\
&+z^{\alpha+2} (-3)2^{\alpha+1}\mathrm{Ste} (\alpha+7)+z^{\alpha+1} 3 \frac{2^{\alpha+1}}{\mathrm{Bi}}\mathrm{Ste} (\alpha-5)+z^{\alpha} 9 \;2^\alpha \mathrm{Ste}+z^4 2 \mathrm{Ste}^2 (1+\alpha)^2\nonumber \\
&+ z^2 (-12)\mathrm{Ste}^2 (\alpha+1)+18\mathrm{Ste}^2=0, \qquad z>0. \label{EcNu1h}
\end{align}
\end{teo} 

\begin{proof}
We shall notice first that the profile chosen (\ref{PerfilT1h}), makes the condition  (\ref{TempFront:1F-pos-tempinfty-A}) to be verified automatically. 
In addition we have
$$\pder[T_{_{1h}}]{x}(x,t)=-t^{\alpha/2} \theta_{_\infty}\left[\frac{A_{_{1h}}}{s_{_{1h}}(t)}+\frac{2B_{_{1h}}}{s_{_{1h}}(t)}\left(1-\frac{x}{s_{1h}(t)}\right) \right],$$
and
$$
\pder[^2T_{_{1h}}]{x^2}(x,t)=t^{\alpha/2}\theta_{_\infty} \frac{2B_{_{1h}}}{s_{_{1h}}^2(t)}.
$$
In virtue of condition (\ref{CondStefan:1F-pos-tempinfty-A}$^\star$), the following equality holds
$$\frac{k}{\gamma s^{\alpha}_{_{1h}}(t)} t^{\alpha} \theta_{_\infty}^2 \frac{A_{_{1h}}^2}{s_{1h}^2 (t)}=a^2 t^{\alpha/2} \theta_{_\infty}\frac{2B_{_{1h}}}{s_{_{1h}}^2(t)}. $$
Consequently
$$s_{_{1h}}(t)= \left( \frac{A_{_{1h}}^2}{2 B_{_{1h}}} \frac{k \theta_{_\infty}}{\gamma a^2}\right)^{1/\alpha} \sqrt{t}.$$
Defining $\nu_{_{1h}}$ such that $\nu_{_{1h}}=\frac{1}{2a}\left( \frac{A_{_{1h}}^2}{2 B_{_{1h}}} \frac{k \theta_{_\infty}}{\gamma a^2}\right)^{1/\alpha}$, we conclude that
\be \label{s1hDemo}
s_{_{1h}}(t)=2a\nu_{_{1h}}\sqrt{t},
\ee
where $\nu_{_{1h}}$ is an unknown that is related with  $A_{_{1h}}$ and $B_{_{1h}}$ in the following way
\be\label{1h}
A_{_{1h}}^2 =\frac{2^{\alpha+1} \nu_{_{1h}}^{\alpha}}{\Ste} B_{_{1h}}.
\ee
Then, condition (\ref{EcCalor:1F-pos-tempinfty-A}$^{\star}$) leads to 
\be\label{2h} 
A_{_{1h}}\left[(\alpha+1)\nu^2_{_{1h}}-1 \right]+B_{_{1h}}\left[\frac{2}{3}(\alpha+1)\nu_{_{1h}}^2-2 \right]=-\frac{2^{\alpha+1}}{\Ste}\nu_{_{1h}}.
\ee
In addition, according to (\ref{FrontFijaConvectiva}) we have
\be\label{3h} 
A_{_{1h}}\left( 1+2\Bt\; \nu_{_{1h}} \right)+2B_{_{1h}}\left(1+\Bt\; \nu_{_{1h}} \right)=2\Bt\; \nu_{_{1h}}.
\ee

Thus, we have obtained three equations (\ref{1h}), (\ref{2h}) and (\ref{3h}), for the three  unknown coefficients $A_{_{1h}}$, $B_{_{1h}}$ and $\nu_{_{1h}}$.

From (\ref{2h}) and (\ref{3h}) we obtain that $A_{_{1h}}$ and $B_{_{1h}}$ are given by (\ref{A1h}) and (\ref{B1h}), respectively.

Then, equation (\ref{1h}) leads to  $\nu_{_{1h}}$ as a positive solution to equation (\ref{EcNu1h}). 
If we denote by $\omega_{_{1h}}=\omega_{_{1h}}(z)$ the left hand side of equation (\ref{EcNu1h}), we have
\be
 \omega_{_{1h}}(0)=18\; \Ste^2>0
\ee
and
\begin{align}
\omega_{_{1h}}(1)&=-\alpha^2\left( 3\; 2^{\alpha}-2\Ste+\tfrac{3}{\Bt}2^{\alpha+1}\right)\Ste-2\alpha\left(3\; 4^{\alpha}+4\Ste^2+\tfrac{21}{\Bt}2^{\alpha-1}-\tfrac{3}{\Bt}2^{\alpha}\Ste \right)\nonumber\\
&-2\left( 3\; 4^{\alpha}+3\; 2^{2+\alpha}\Ste-4\Ste^2\right)+\tfrac{3}{\Bt}\left( 2^{2\alpha+3}-2^{3+\alpha}\Ste\right).
\end{align}
It can be noticed that if $0<\Ste<1$ and $\alpha\geq 0$ we have
\begin{align*}
& 3\; 2^{\alpha}-2\Ste+\frac{3}{\Bt}2^{\alpha+1}>0,\\
&  3\; 4^{\alpha}+3\; 2^{2+\alpha}\Ste-4\Ste^2>0,
\end{align*}
and
$$3\; 4^{\alpha}+4\Ste^2+\frac{21}{\Bt}2^{\alpha-1}-\frac{3}{\Bt}2^{\alpha}\Ste=3\; 4^{\alpha}+4\Ste^2+ \frac{3}{\Bt} 2^{\alpha }\left(\frac{7}{2}-\Ste \right)>0.$$
As $ 2^{2\alpha+3}-2^{3+\alpha}\Ste=2^{\alpha}2^3 (2^\alpha-\Ste)>0$, there exists a large enough Biot number $\Bt$ that makes $\omega_{_{1h}}(1)<0$. In consequence, there will exists at least one solution to equation (\ref{EcNu1h}).

\end{proof}

With the aim of testing the accuracy of the classical heat balance integral method and taking into account that the exact free boundary $s_{h}(t)=2a\nu_{h}\sqrt{t}$ and the approximate one is given by $s_{_{1h}}(t)=2a\nu_{_{1h}}\sqrt{t}$ we are going to compare graphically only the coefficients $\nu_{h}$ with  $\nu_{_{1h}}$ for different values of $\Bt$ and $\alpha$, fixing $\Ste=0.5$ (see Figure \ref{Fig:Nu1hVsNuh}).

\begin{figure}[h!!]
\begin{center}
\includegraphics[scale=0.4]{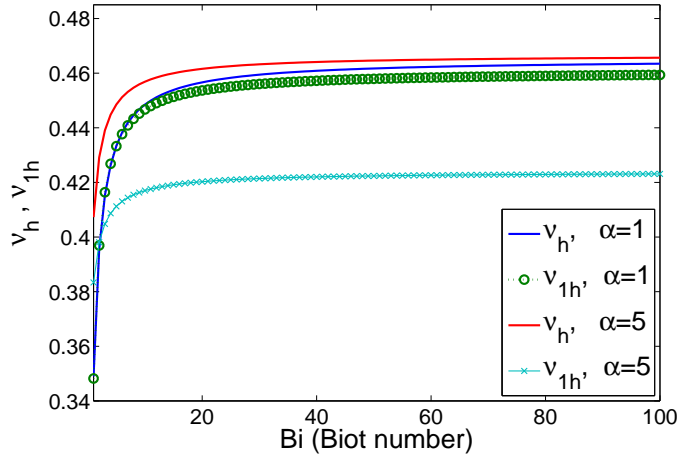}
\end{center}
\caption{{\footnotesize  Plot of $\nu_{h}$ and $\nu_{_{1h}}$ against $\Bt$ for $\alpha=1$ or 5 and $\Ste=0.5$}}
\label{Fig:Nu1hVsNuh}
\end{figure}

The modified integral balance method defines a new approximated problem for (P$_{h}$) that will be called as \textbf{problem (P$_{_{2h}}$)} and which consists in finding the free boundary  $s_{_{2h}}=s_{_{2h}}(t)$  and the temperature $T_{_{2h}}=T_{_{2h}}(x,t)$ in  $0<x<s_{_{2h}}(t)$ such that equations (\ref{EcCalor:1F-pos-tempinfty-A}$^\star$), (\ref{FrontFijaConvectiva}),(\ref{TempFront:1F-pos-tempinfty-A})-(\ref{FrontInicial:1F-pos-tempinfty-A}) are satisfied.

Once again assuming a quadratic profile in space as (\ref{Perfil}) for the temperature $T_{_{2h}}$ we can state the following results

\begin{teo}\label{TeoP2h}
Given $\mathrm{Ste}>0$ and $\alpha\geq 0$, there exists a unique solution to the problem $\mathrm{(P_{_2})}$  which is given by
\begin{eqnarray}
T_{_{2h}}(x,t)&=&t^{\alpha/2}\left[ A_{_{2h}}\theta_{_\infty}\left(1-\frac{x}{s_{_{2h}}(t)}\right) +B_{_{2h}}\theta_{_\infty} \left(1-\frac{x}{s_{_{2h}}(t)}\right)^{2}\right],\label{PerfilT2h} \\
s_{_{2h}}(t)&=& 2a\nu_{_{2h}} \sqrt{t},
\end{eqnarray}
where the constants $A_{_{2h}}$ and  $B_{_{2h}}$ are given by
\begin{align} 
A_{_{2h}}&=\frac{6\mathrm{Ste}-2\mathrm{Ste}\; \nu^2_{_{2h}}(\alpha+1)-\frac{3}{\mathrm{Bi}} 2^{\alpha+1} \nu_{_{2h}}^{\alpha+1}-3\; 2^{\alpha+1}\nu_{_{2h}}^{\alpha+2} }{\mathrm{Ste}\left[\nu_{_{2h}}^2 (\alpha+1)+\frac{2}{\mathrm{Bi}} \nu_{_{2h}}(\alpha+1)+3 \right]},\qquad\label{A2h} \\[0.2cm]
B_{_{2h}}&= \frac{-3\mathrm{Ste}+3\mathrm{Ste}\; \nu_{_{2h}}^2(\alpha+1)+\frac{3}{\mathrm{Bi}} 2^{\alpha} \nu_{_{2h}}^{\alpha+1}+3\;2^{\alpha+1} \nu_{_{2h}}^{\alpha+2} }{\mathrm{Ste}\left[\nu_{_{2h}}^2 (\alpha+1)+\frac{2}{\mathrm{Bi}} \nu_{_{2h}}(\alpha+1)+3 \right]},\label{B2h}
\end{align}
and where the coefficient $\nu_{_{2h}}$ is the unique solution to the following equation
\begin{align}
&z^{\alpha+4} 2^{\alpha} (\alpha+1)+z^{\alpha+3}\; \tfrac{2^{\alpha+1}}{\mathrm{Bi}}(\alpha+1)+z^{\alpha+2}3\; 2^{\alpha+1}\nonumber\\
&+z^{\alpha+1}3\tfrac{2^{\alpha}}{\mathrm{Bi}}+z^2 \mathrm{Ste} (\alpha+1)-3\mathrm{Ste}=0, \qquad z>0. \label{EcNu2h}
\end{align}
\end{teo}
\begin{proof}
It is clear immediate that the chosen profile temperature leads the condition (\ref{TempFront:1F-pos-tempinfty-A}) to be automatically verified.
From condition (\ref{CondStefan:1F-pos-tempinfty-A})  we obtain
\be
-kt^{\alpha/2}\theta_{_\infty}\frac{A_{_{2h}}}{s_{_{2h}}(t)}=-\gamma s_{_{2h}}^{\alpha}(t)\dot s_{_{2h}}(t).
\ee
Therefore
\be
s_{_{2h}}(t)=\left( \frac{(\alpha+2)}{(\frac{\alpha}{2}+1)} \frac{k\theta_{_\infty}}{\gamma}A_{_{2h}}\right)^{1/(\alpha+2)} \sqrt{t}.
\ee
Introducing the new coefficient $\nu_{_{2h}}$ such that $ \nu_{_{2h}}=\frac{1}{2a}\left( \frac{(\alpha+2)}{(\frac{\alpha}{2}+1)} \frac{k\theta_{_\infty}}{\gamma}A_{_{2h}}\right)^{1/(\alpha+2)}$, the free boundary can be expressed as
\be 
s_{_{2h}}(t)=2a\; \nu_{_{2h}}\sqrt{t},
\ee
where the following equality holds
\be\label{Ec1:P2h}
A_{_{2h}}=\frac{2^{\alpha+1} \nu_{_{2h}}^{\alpha+2}}{\Ste} .
\ee
The  convective boundary condition at $x=0$, i.e. condition (\ref{FrontFijaConvectiva}), leads to
\be \label{Ec2:P2h}
A_{_{2h}}(1+2\Bt\; \nu_{_{2h}})+2B_{_{2h}}(1+\Bt\;\nu_{_{2h}})=2\Bt\;\nu_{_{2h}}.
\ee
In addition, from (\ref{EcCalor:1F-pos-tempinfty-A}$^{\star}$) it results that
\be \label{Ec3_P2h}
A_{_{2h}}\left( (\alpha+1)\nu_{_{2h}}^2-1 \right) +B_{_{2h}} \left( \tfrac{2}{3} (\alpha+1)\nu_{_{2h}}^2-2\right)=\tfrac{-2^{\alpha+1}\nu_{_{2h}}^{\alpha+2}}{\Ste}.
\ee
Taking into account equations (\ref{Ec1:P2h})-(\ref{Ec3_P2h}) we obtain that  $A_{_{2h}}$ y $B_{_{2h}}$ can be given as functions of $\nu_{_{2h}}$ through formulas (\ref{A2h}) and (\ref{B2h}), respectively. Moreover, we get that $\nu_{_{2h}}$ must be a solution to equation (\ref{EcNu2h}). To finish the proof   it remains to show that the equation  (\ref{EcNu2h}) has a unique positive solution. If we define the function  $w_{_{2h}}=w_{_{2h}}(z)$ as the  left hand side of equation (\ref{EcNu2h}) we have that
$$w_{_{2h}}(0)=-3\Ste<0,\qquad w_{_{2h}}(+\infty)=+\infty, \qquad \frac{\dee w_{_{2h}}}{\dee z}(z)>0,\quad \forall z>0.$$ 
So we conclude that $w_{_{2h}}$ has a unique positive root.
\end{proof}

In what follows, we will show that the unique solution to the problem  (P$_{_{2h}}$) converges to the unique solution to the problem (P$_{_2}$) when $h\to \infty$.

\begin{teo}\label{ConvergenciaP2h}
The solution to problem (P$_{_{2h}}$) given in Theorem \ref{TeoP2h} converges to the solution to problem (P$_{_2}$) given by Theorem \ref{TeoP2} when the coefficient $h$,  that characterizes the heat transfer in the fixed boundary, goes to infinity
\end{teo} 
\begin{proof}
The free boundary of the problem (P$_{_{2h}}$) is characterized by a dimensionless coefficient $\nu_{_{2h}}$ which is the unique positive root of the function $\omega_{_{2h}}=\omega_{_{2h}}(z)$ defined as the left hand side of equation (\ref{EcNu2h}).  
On one hand, we can notice that if $h_1<h_2$ then  $\omega_{\mathrm{2h_1}}(z)>\omega_{\mathrm{2h_2}}(z)$ and consequently their unique positive root verify $\nu_{_\mathrm{2h_1}}<\nu_{_\mathrm{2h_2}}$.\\
On the other hand, if we define $\omega_{_2}=\omega_{_2}(z)$ as the left hand side of equation (\ref{EcNu2}), we get
$$\omega_{_{2h}}(z)-\omega_{_2}(z)= z^{\alpha+3}\frac{2^{\alpha+1}}{\mathrm{Bi}}(\alpha+1)+z^{\alpha+1}3\frac{2^{\alpha}}{\mathrm{Bi}}>0, \quad \forall z>0.$$
Therefore $\lbrace \nu_{_h}\rbrace_{_h}$ is increasing  and bounded from above by $\nu.$

In addition, it is easily seen that when $h\to\infty$, or equivalently when $\Bt\to\infty$, we obtain $\omega_{2h}\to\omega_2$ and so $\nu_{_{2h}}\to\nu_{_2}$.
Therefore it is obtained that  $s_{_{2h}}(t)\to s_{_2}(t)$, for every $t>0$. Showing that $A_{_{2h}}\to A_{_2}$ and $B_{_{2h}}\to B_{_2}$  we get $T_{_{2h}}(x,t)\to T_{_2}(x,t)$ when $h\to\infty$ for every $t>0$ and $0<x<s_{_2}(t)$.
\end{proof}

In Figure \ref{Fig:Nu2hVsNuh}  we compare graphically, for different values of $\Bt>1$, the coefficient $\nu_{_{2h}}$ that characterizes the free boundary $s_{_{2h}}$ with the coefficient  $\nu_{h}$ that characterizes the exact free boundary  $s_{h}$, for different values of  $\alpha$, fixing $\Ste=0.5$. We shall notice that when the Biot number increases then the value of $\nu_{_{2h}}$ gets closer to the value of $\nu_{_2}$.

\begin{figure}[h!!]
\begin{center}
\includegraphics[scale=0.4]{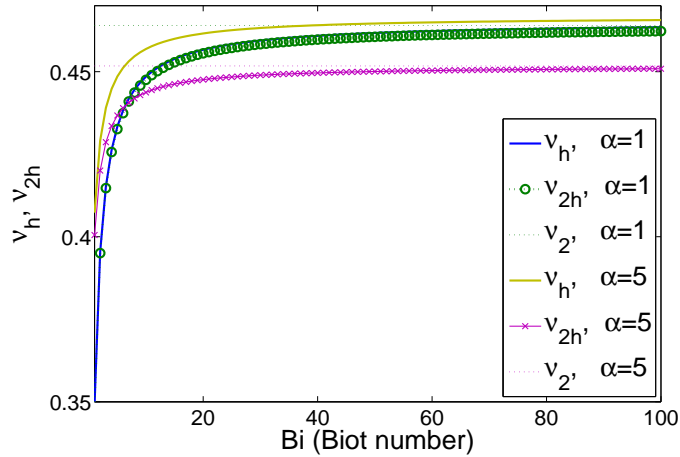}
\end{center}
\caption{{\footnotesize  Plot of $\nu_{h}$ and $\nu_{_{2h}}$ against $\Bt$ for $\alpha=1$ or 5 and $\Ste=0.5$}}\label{Fig:Nu2hVsNuh}
\end{figure}

\newpage

Lastly we will turn to the refined integral method applied to problem (P$_{h}$). We define a new approximate\textbf{ problem (P$_{_{3h}}$) }which consists in finding the free boundary $s_{_{3h}}=s_{_{3h}}(t)$ and the temperature $T_{_{3h}}=T_{_{3h}}(x,t)$ in $0<x<s_{_{3h}}(t)$  such that equations (\ref{EcCalor:1F-pos-tempinfty-A}$^\dag$), (\ref{FrontFijaConvectiva}),(\ref{TempFront:1F-pos-tempinfty-A})-(\ref{FrontInicial:1F-pos-tempinfty-A}) are verified.

Provided that $T_{_{3h}}$ adopts a profile like (\ref{Perfil}) we state the following theorem

\begin{teo}\label{TeoP3h}
Let $0<\mathrm{Ste}<1$, $\alpha\geq 0$ and $\mathrm{Bi}\geq 0$, then there exists a unique solution to problem $\mathrm{(P_{_{3h}})}$ which is given by
\begin{eqnarray}
T_{_{3h}}(x,t)&=&t^{\alpha/2}\left[ A_{_{3h}}\theta_{_\infty}\left(1-\frac{x}{s_{_{3h}}(t)}\right) +B_{_{3h}}\theta_{_\infty} \left(1-\frac{x}{s_{_{3h}}(t)}\right)^{2}\right],\label{PerfilT3h} \\
s_{_{3h}}(t)&=& 2a\nu_{_{3h}} \sqrt{t},
\end{eqnarray}
where the constants $A_{_{3h}}$ and  $B_{_{3h}}$ are defined by
\begin{align} 
A_{_{3h}}&=\frac{12\nu_{_{3h}} \left(1-\nu_{_{3h}}^2 \left( \frac{\alpha}{2}+\frac{1}{3}\right) \right)  }{2\alpha \nu_{_{3h}}^3+\left( \frac{5\alpha+2}{\mathrm{Bi}}\right)\nu_{_{3h}}^2+\frac{6}{\mathrm{Bi}}+12\nu_{_{3h}} },\label{A3h}\\
B_{_{3h}}&=\frac{12\nu_{_{3h}}^3 \left(\frac{2}{3}\alpha+\frac{1}{3}\right)  }{2\alpha \nu_{_{3h}}^3+\left( \frac{5\alpha+2}{\mathrm{Bi}}\right)\nu_{_{3h}}^2+\frac{6}{\mathrm{Bi}}+12\nu_{_{3h}} },\label{B3h}
 \end{align}
and where  $\nu_{_{3h}}$ is the unique solution to the following equation
\begin{align}
&z^{\alpha+4} 2^{\alpha+1} \alpha+z^{\alpha+3}\left(\tfrac{2^{\alpha}(2+5\alpha)}{\mathrm{Bi}}\right)  +z^{\alpha+2}3\; 2^{\alpha+2}+z^{\alpha+1} \tfrac{3\; 2^{\alpha+1}}{\mathrm{Bi}}\nonumber\\
&+z^2 \mathrm{Ste} (2+3\alpha)-6\mathrm{Ste}=0, \qquad z>0. \label{EcNu3h}
\end{align}

\end{teo}

\begin{proof}
The proof is similar to the one given in Theorem \ref{TeoP2h}. The only difference lies in the fact that equation  (\ref{EcCalor:1F-pos-tempinfty-A}$^{\dag}$) is equivalent to
\be
\nu_{_{3h}}^2 \left[A_{_{3h}}\left( \tfrac{1}{3}+\tfrac{2}{3}\alpha\right)+B_{_{3h}}\left( \tfrac{1}{3}+\tfrac{\alpha}{2}\right) \right]=B_{_{3h}} 
\ee

\end{proof}

The approximated problem  (P$_{_{3h}}$) obtained when applying the refined integral method verify the same convergence property than the exact problem (P$_{h}$). 

\begin{teo} \label{ConvergenciaP3h}
The unique solution to problem (P$_{_{3h}}$) given by Theorem \ref{TeoP3h} converges to the unique solution to problem (P$_{_3}$), given by Theorem \ref{TeoP3}, when the coefficient that charaterizes the heat transfer at the fixed face  $h$ goes to infinity.
\end{teo} 
\begin{proof}
The proof is analogous to the proof given in Theorem \ref{ConvergenciaP2h}.
\end{proof}

\medskip
In Figure \ref{Fig:Nu3hVsNuh}  we compare graphically, for different values of $\Bt>1$, the coefficient $\nu_{_{3h}}$ that characterizes the approximate free boundary $s_{_{3h}}$ with the coefficient $\nu_{h}$  corresponding to the exact free boundary $s_{h}$, for different values of $\alpha$ fixing $\Ste=0.5$. Once again, as  $\Bt$ increases, the value $\nu_{_{3h}}$ becomes closer to the value $\nu_{_3}$

\begin{figure}[h!!]
\begin{center}
\includegraphics[scale=0.4]{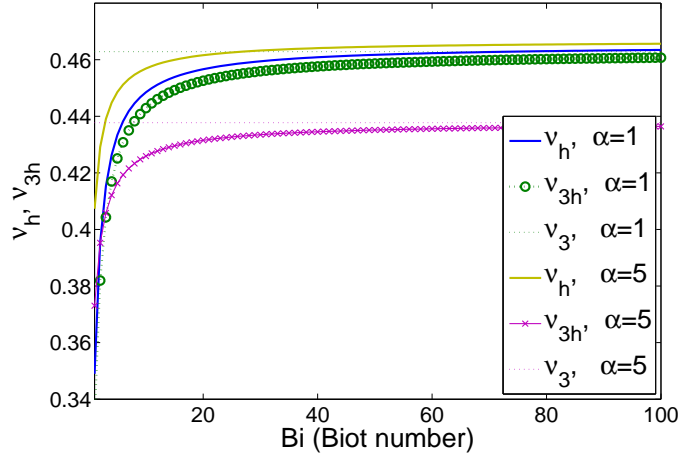}
\end{center}
\caption{{\footnotesize  Plot of $\nu_{h}$ and $\nu_{_{3h}}$ against $\Bt$ for $\alpha=1$ or 5 and $\Ste=0.5$}}\label{Fig:Nu3hVsNuh}
\end{figure}
\newpage

\subsection{Comparisons between the approximate solutions and the exact one}

\medskip

In this section we are going to compare the exact solution to the problem with a convective condition at the fixed face (P$_{h}$) with the approximate solutions obtained by applying the integral balance methods proposed in the previous sections.

For each method, we have defined a new problem (P$_{_{_{ih}}}$), ${i}=1,2,3$  and we have compared graphically the coefficient $\nu_{_{_{ih}}}$  that characterizes each free boundary  $s_{_{_{ih}}}$, with the coefficient $\nu_{h}$  that corresponds to the exact free boundary $s_{h}$.

The goal is to compare numerically the coefficient $\nu_{h}$ given by (\ref{EcNuh}) with the approximate coefficients $\nu_{_{1h}}$, $\nu_{_{2h}}$  and $\nu_{_{3h}}$ given by (\ref{EcNu1h}), (\ref{EcNu2h}) and (\ref{EcNu3h}), respectively.

In order that the comparisons be more representative, in Tables  \ref{Tabla:NuihVsNu1h}-\ref{Tabla:NuihVsNu3h} we show the exact value $\nu_{h}$, the approximate value $\nu_{_{_{ih}}}$ and the percentaje error committed in each case $E(\nu_{_{_{ih}}})=100\left\vert\frac{\nu_{h}-\nu_{{_{ih}}}}{\nu_{h}} \right\vert$, ${i}=1,2,3$ for different values of $\Bt$ and $\alpha$ fixing $\Ste=0.5$.

\begin{table}
\small
\caption{{\footnotesize Dimensionless coefficients of the free boundaries and their percentage relative error for  $\alpha=0$ and $\Ste=0.5$.}}
\label{Tabla:NuihVsNu1h}  
\begin{center}
\begin{tabular}{cc|cc|cc|cc}
\hline
\Bt   & $\nu_{h}$     &   $\nu_{_{1h}}$ & $E_{\text{rel}}(\nu_{_{1h}})$       &   $ \nu_{_{2h}} $  & $E_{\text{rel}}(\nu_{_{2h}})$      &   $\nu_{_{3h}}$ & $E_{\text{rel}}(\nu_{_{3h}})$   \\
\hline
   1&    0.2926 &   0.2966  &  1.3828  \% &  0.2937 &   0.3939 \% &   0.2899 &   0.9103 \% \\
   10 &    0.4422 &   0.4681  &  5.8548  \% &  0.4484 &   1.4111 \% &   0.4545 &   2.7969 \% \\
   20 &    0.4533 &   0.4776  &  5.3525  \% &  0.4602 &   1.5151 \% &   0.4672 &   3.0744 \% \\
   30 &    0.4571 &   0.4807  &  5.1622  \% &  0.4642 &   1.5514 \% &   0.4716 &   3.1679 \% \\
   40 &    0.4590 &   0.4822  &  5.0628  \% &  0.4662 &   1.5699 \% &   0.4738 &   3.2148 \% \\
   50 &    0.4601 &   0.4832  &  5.0019  \% &  0.4674 &   1.5811 \% &   0.4751 &   3.2430 \% \\
   60 &    0.4609 &   0.4838  &  4.9606  \% &  0.4682 &   1.5886 \% &   0.4759 &   3.2618 \% \\
   70 &    0.4615 &   0.4842  &  4.9309  \% &  0.4688 &   1.5940 \% &   0.4766 &   3.2752 \% \\
   80 &    0.4619 &   0.4845  &  4.9085  \% &  0.4693 &   1.5980 \% &   0.4771 &   3.2853 \% \\
   90 &    0.4622 &   0.4848  &  4.8909  \% &  0.4696 &   1.6012 \% &   0.4774 &   3.2932 \% \\
  100 &    0.4625 &   0.4850  &  4.8768  \% & 0.4699  &  1.6037  \% &   0.4777 &   3.2994 \% \\

    \hline
\end{tabular}
\end{center}
\end{table}

\begin{table}
\small
\caption{{\footnotesize Dimensionless coefficients of the free boundaries and their percentage relative error for  $\alpha=5$ and $\Ste=0.5$.}}
\label{Tabla:NuihVsNu2h}  
\begin{center}
\begin{tabular}{cc|cc|cc|cc}
\hline
\Bt    & $\nu_{h}$     &   $\nu_{_{1h}}$ & $E_{\text{rel}}(\nu_{_{1h}})$       &   $ \nu_{_{2h}} $  & $E_{\text{rel}}(\nu_{_{2h}})$      &   $\nu_{_{3h}}$ & $E_{\text{rel}}(\nu_{_{3h}})$   \\
\hline
   1&    0.3274 &   0.3293  &   0.5908  \% &   0.3280   & 0.1779 \% &    0.3160 &   3.4746 \% \\
   10 &    0.4459 &   0.4551  &  2.0484   \% &  0.4480    & 0.4543 \% &    0.4474 &   0.3370 \% \\
   20 &    0.4553 &   0.4631  &  1.7173   \% &  0.4574    & 0.4798 \% &    0.4583 &   0.6724 \% \\
   30 &    0.4585 &   0.4657  &  1.5912   \% &  0.4607    & 0.4886 \% &    0.4621 &   0.7874 \% \\
   40 &    0.4601 &   0.4671  &  1.5250   \% &  0.4623    & 0.4931 \% &    0.4640 &   0.8456 \% \\
   50 &    0.4610 &    0.4679 &   1.4844  \% &   0.4633   & 0.4958 \% &    0.4651 &   0.8807 \% \\
   60 &    0.4617 &   0.4684  &  1.4569    \% & 0.4640    & 0.4976 \% &    0.4659 &   0.9042 \% \\
   70 &    0.4622 &   0.4688  &  1.4370   \% &  0.4645    & 0.4989 \% &    0.4664 &   0.9210 \% \\
   80 &    0.4625 &   0.4691  &  1.4220  \% &   0.4648    & 0.4999  \% &   0.4668 &   0.9336 \% \\
   90 &    0.4628 &   0.4693  &  1.4103  \% &   0.4651    & 0.5006 \% &    0.4672 &   0.9434 \% \\
  100 &    0.4630 &   0.4695  &  1.4009   \% &  0.4653    & 0.5012  \% &   0.4674 &   0.9513 \% \\
    \hline
\end{tabular}
\end{center}
\end{table}

\begin{table}
\small
\caption{{\footnotesize Dimensionless coefficients of the free boundaries and their percentage relative error for  $\alpha=0.5$ and $\Ste=0.5$.}}
\label{Tabla:NuihVsNu3h}  
\begin{center}
\begin{tabular}{cc|cc|cc|cc}
\hline
\Bt   & $\nu_{h}$     &   $\nu_{_{1h}}$ & $E_{\text{rel}}(\nu_{_{1h}})$       &   $ \nu_{_{2h}} $  & $E_{\text{rel}}(\nu_{_{2h}})$      &   $\nu_{_{3h}}$ & $E_{\text{rel}}(\nu_{_{3h}})$   \\
\hline
   1 &    0.4073  &  0.3834 &   5.8702  \% &    0.4005 &   1.6647 \% &    0.3730  &  8.4069 \% \\
   10  &    0.4569  &  0.4170 &   8.7307  \% &    0.4437 &   2.8806 \% &    0.4259  &  6.7799 \% \\
   20  &    0.4616  &  0.4203 &   8.9507  \% &    0.4476 &   3.0301 \% &    0.4315  &  6.5196 \% \\
   30  &    0.4632  &  0.4214 &   9.0256  \% &    0.4489 &   3.0845 \% &    0.4335  &  6.4217 \% \\
   40  &    0.4641  &  0.4220 &   9.0633  \% &    0.4496 &   3.1126 \% &    0.4345  &  6.3703 \% \\
   50  &    0.4646  &  0.4224 &   9.0861  \% &    0.4501 &   3.1298 \% &    0.4351  &  6.3387 \% \\
   60  &    0.4649  &  0.4226 &   9.1012  \% &    0.4503 &   3.1414 \% &    0.4356  &  6.3173 \% \\
   70  &    0.4652  &  0.4228 &   9.1121  \% &    0.4505 &   3.1497 \% &    0.4359  &  6.3018 \% \\
   80  &    0.4654  &  0.4229 &   9.1203  \% &    0.4507 &   3.1560 \% &    0.4361  &  6.2901 \% \\
   90  &    0.4655  &  0.4230 &   9.1266  \% &    0.4508 &   3.1609 \% &    0.4363  &  6.2809 \% \\
  100  &    0.4656  &  0.4231 &   9.1317  \% &    0.4509 &   3.1649 \% &    0.4364  &  6.2736 \% \\

    \hline
\end{tabular}
\end{center}
\end{table}
\newpage

From the above tables we can deduce that for $\alpha=0.5$, the percentage error committed is smaller than for the other cases. In all cases, as it happened with the problem (P), the method with best accuracy for approximating the problem (P$_{h}$) is the modified integral method, i.e. the best approximate problem is given by (P$_{_{2h}}$).

We can also compare the exact temperature $T_{h}$ with the approximate ones $T_{_{_{ih}}}$, ${i}=1,2,3$, given by (\ref{PerfilT1h}), (\ref{PerfilT2h}) and (\ref{PerfilT3h}), respectively. In Figures  (11)-(14) we show a color map for $\alpha=5$, $\Ste=0.5$, $\theta_{_\infty}=30, a=1$

\begin{figure}[h!]\centering
   \begin{minipage}{0.49\textwidth}
   \begin{center}
    \includegraphics[scale=0.3]{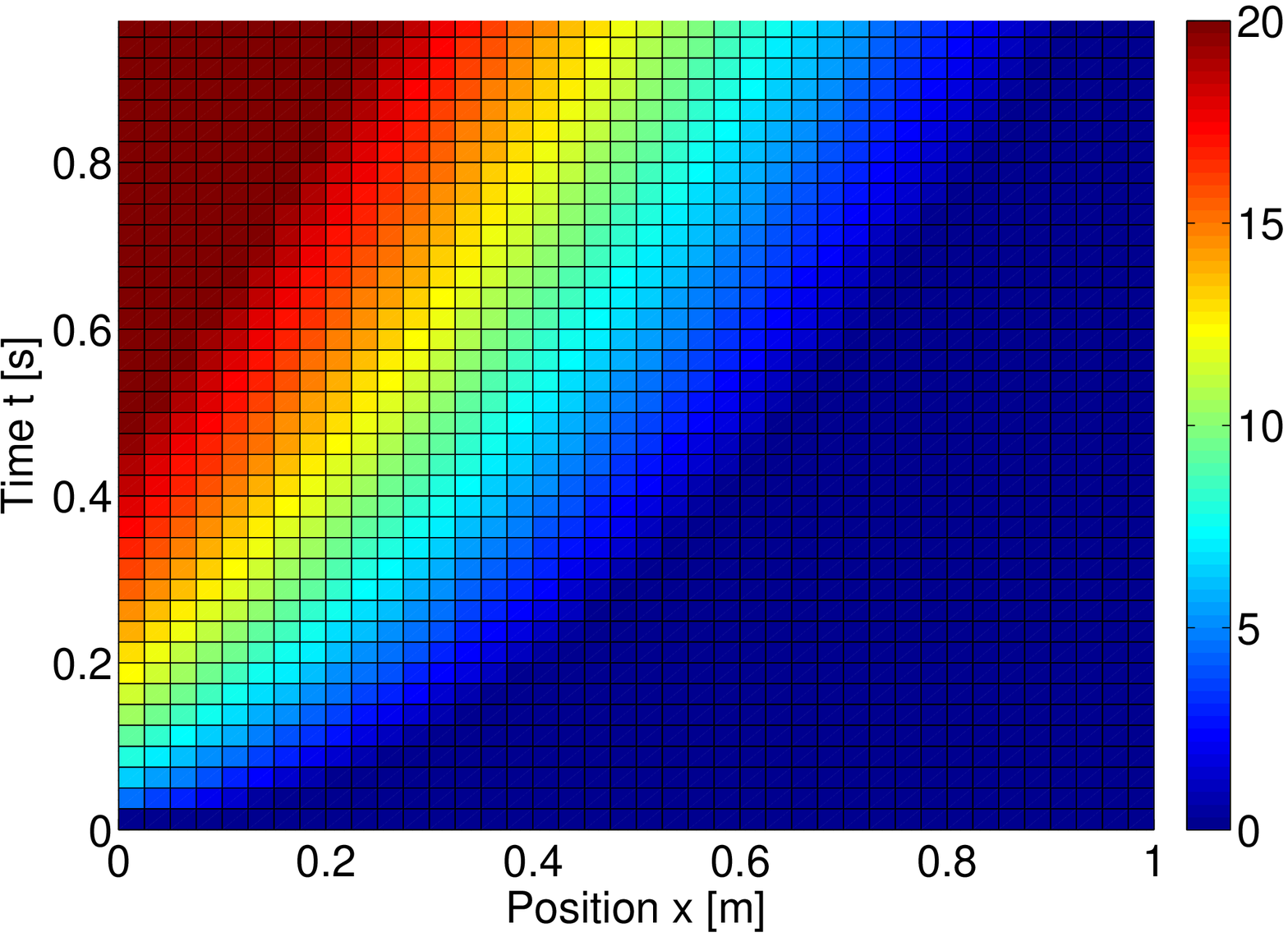}
     \caption{{\footnotesize Colour map for $T_{h}$ }}
   \end{center}
\label{Fig:ExactaConv}
   \end{minipage}
   \begin {minipage}{0.49\textwidth}
\begin{center}
	\includegraphics[scale=0.3]{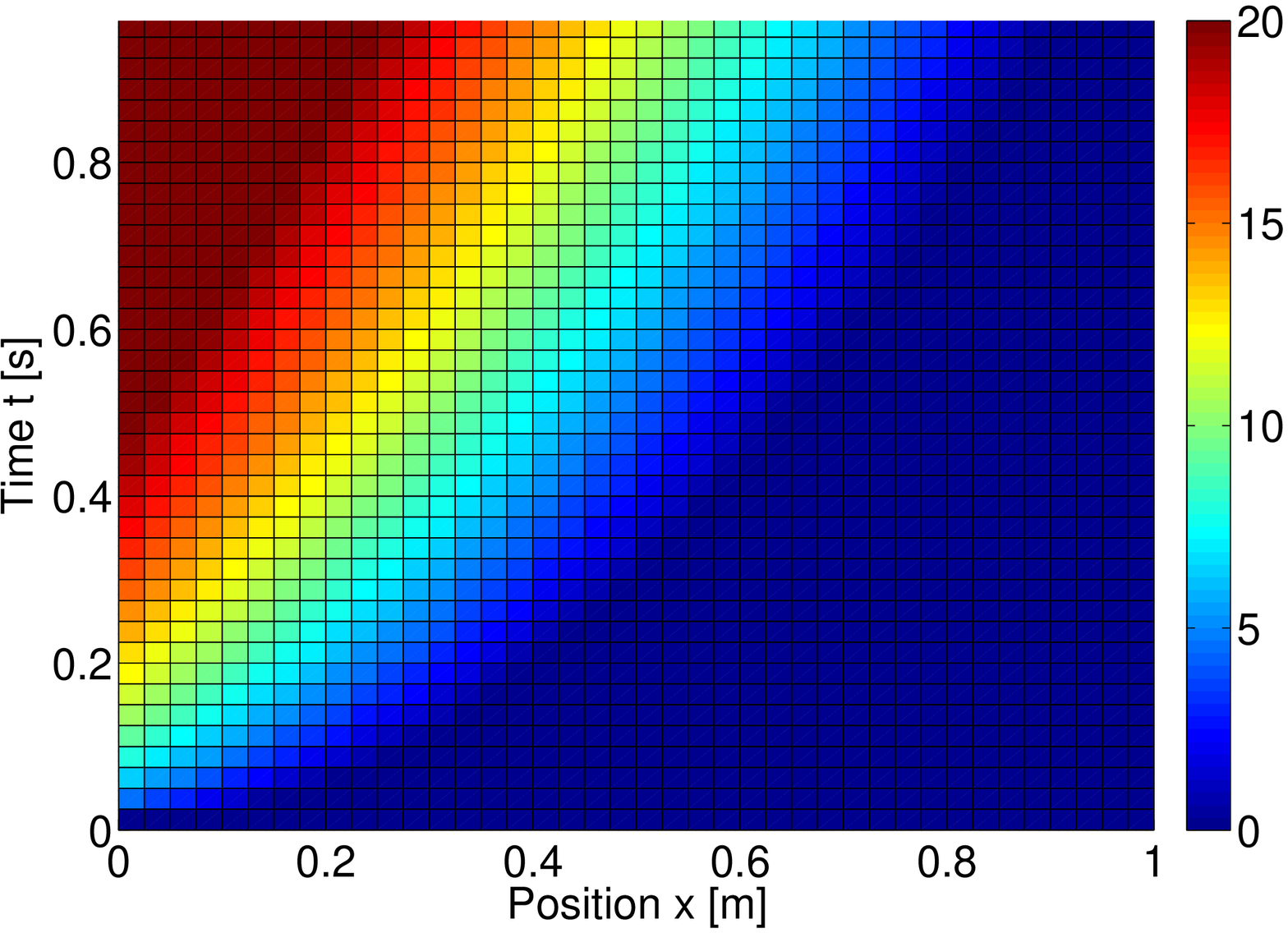}
	  \caption{ {\footnotesize  Colour map for $T_{_{1h}}$}}
\end{center}
   \label{Fig:BICConv}
   \end{minipage}
\end{figure}

\begin{figure}[h!]\centering
   \begin{minipage}{0.49\textwidth}
   \begin{center}
    \includegraphics[scale=0.3]{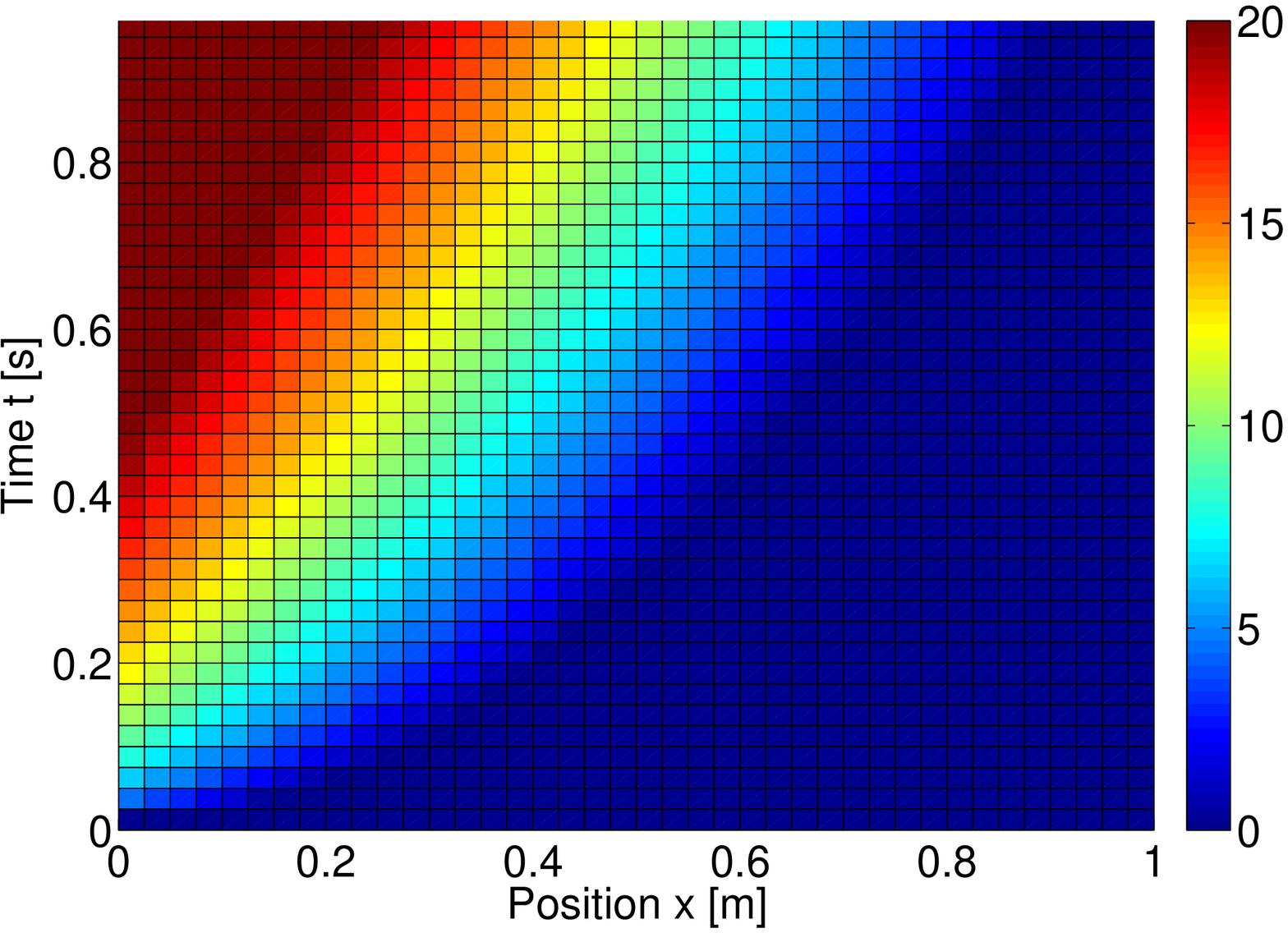}
     \caption{{\footnotesize Colour map for $T_{_{2h}}$ }}
   \end{center}
\label{Fig:BIMConv}
   \end{minipage}
   \begin {minipage}{0.49\textwidth}
\begin{center}
	\includegraphics[scale=0.3]{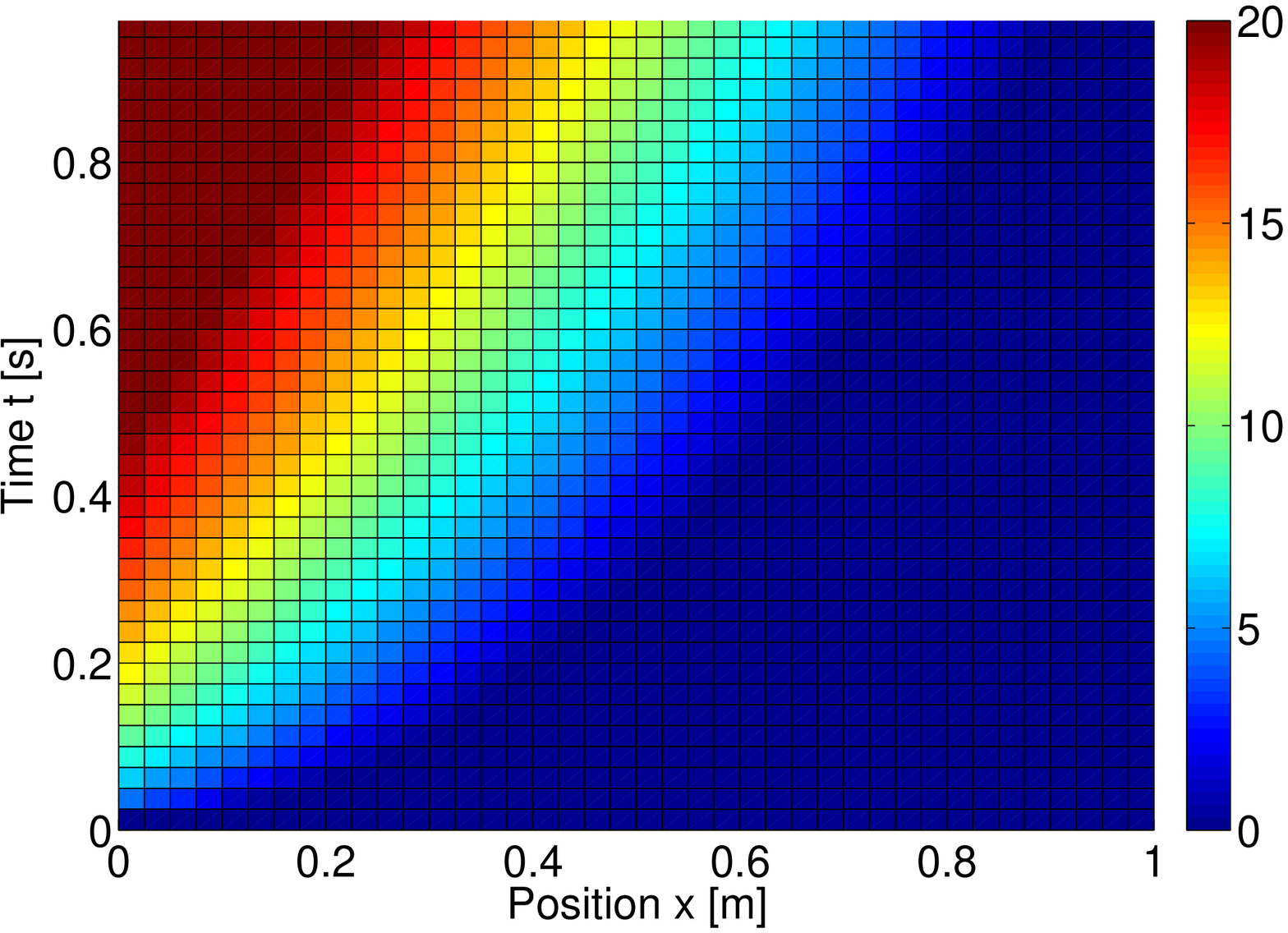}
	  \caption{ {\footnotesize  Colour map for $T_{_{3h}}$}}
\end{center}
   \label{Fig:RIMConv}
   \end{minipage}
\end{figure}

\newpage

\section{Minimising the least-squares error in the heat balance integral method}

In this section, we are going to analyse the least-squares error that we commit when assuming a quadratic profile in space. If we have an approximate solution for the heat equation given by $\hat{T}$, $\hat{s}$ such that 
\begin{equation}\label{PerfilQuadTemp}
\hat{T}(x,t)=t^{\alpha/2}\theta_{\infty}\left[\hat{A}\left( 1-\frac{x}{\hat{s}(t)}\right)+\hat{B}\left(  1-\frac{x}{\hat{s}(t)}\right)^2\right], 
\end{equation}
with adequate coefficients $\hat{A}, \hat{B}$ and $\hat{s}$, then we can measure how far we are from the heat equation by computing the least-squares error  (see \cite{RiMyMc2019}) given by 
\begin{equation}\label{ErrorTemp}
E=\int\limits_0^{\hat{s}(t)} \left(\pder[\hat{T}]{t}(x,t)-a^2\pder[^2 \hat{T}]{x^2}(x,t) \right)^2 \dee x
\end{equation}
Taking into account that
\begin{align}
\pder[\hat{T}]{t}(x,t)&=\frac{\alpha}{2} t^{\alpha/2-1} \theta_{_\infty}\left[ \hat{A}\left(1-\frac{x}{\hat{s}(t)}\right)+\hat{B}\left(1-\frac{x}{\hat{s}(t)}\right)^2\right]\nonumber\\
& +t^{\alpha/2} \frac{\dot{\hat{s}}(t) }{\hat{s}^2(t)} x \theta_{_\infty} \left[ \hat{A}+2\hat{B}\left(1-\frac{x}{\hat{s}(t)}\right)\right]
\end{align}
and
\begin{align}
\pder[^2\hat{T}]{x^2}(x,t)=t^{\alpha/2}\frac{2\hat{B}\theta_{_\infty}}{\hat{s}^2(t)}
\end{align}
we get 
\begin{align}
E&=\tfrac{\alpha^2}{4} \theta_{_\infty}^2 t^{\alpha-2}\left(\tfrac{\hat{A}^2}{3}+\tfrac{\hat{A}\hat{B}}{2}+\tfrac{\hat{B}^2}{5} \right)+t^{\alpha} \theta_{_\infty}^2 \tfrac{\dot{\hat{s}}^2(t)}{\hat{s}^2(t)} \left( \tfrac{\hat{A}^2}{3}+\tfrac{\hat{A} \hat{B}}{3}+\tfrac{2\hat{B}^2}{15}\right)\nonumber\\
&+ 4t^{\alpha}a^4 \theta_{_\infty}^2  \tfrac{\hat{B}^2}{\hat{s}^4(t)}+\alpha \theta_{_\infty}^2 t^{\alpha-1}\tfrac{\dot{\hat{s}}(t)}{\hat{s}(t)} \left( \tfrac{\hat{A}^2}{6}+\tfrac{\hat{A} \hat{B}}{4}+\tfrac{\hat{B}^2}{10}\right)-2\alpha a^2\theta_{_\infty}^2 t^{\alpha-1} \tfrac{\hat{B}}{\hat{s}^2(t)}\left( \tfrac{\hat{A}}{2}+\tfrac{\hat{B}}{3}\right)\nonumber\\
&-4a^2 \theta_{_\infty}^2 t^{\alpha} \tfrac{\dot{\hat{s}}(t)}{\hat{s}^3(t)} \hat{B}\left(\tfrac{\hat{A}}{2}+\tfrac{\hat{B}}{3}\right)\label{ErrorCuadraticoGenerico}
\end{align}
In case that the free boundary $\hat{s}(t)=2a\xi\sqrt{t}$ with $\xi>0$, by simple computations, the least-squares error becomes $E=E(\xi)$, given by the following expression:
\begin{align}
E(\xi)& =t^{\alpha-2} \tfrac{\theta_{_\infty}^2}{\xi^4} \left[ \tfrac{\xi^4}{4}\left( \alpha^2 \left( \tfrac{\hat{A}^2}{3}+\tfrac{\hat{A}\hat{B}}{2}+\tfrac{\hat{B}^2}{5} \right)+2\alpha \left(\tfrac{\hat{A}^2}{6}+\tfrac{\hat{A}\hat{B}}{4}+\tfrac{\hat{B}^2}{10} \right)+\tfrac{\hat{A}^2}{3}+\tfrac{\hat{A}\hat{B}}{3}+\tfrac{2\hat{B}^2}{15}        \right)\right.\nonumber\\
&\left.-\tfrac{\xi^2}{2}\hat{B}(\alpha+1) \left( \tfrac{\hat{A}}{2}+\tfrac{\hat{B}}{3}\right)+\tfrac{\hat{B}^2}{4}
\right]\label{ErrorCuadraticoGenerico2}
\end{align}

Let us then define a new approximate\textbf{ problem (P$_{_4}$)} for the problem (P) that consists in finding the free boundary $s_{_4}=s_{_4}(t)$ and the temperature $T_{_4}=T_{_4}(x,t)$ in the domain $0<x<s_{_4}(t)$ given by the profile (\ref{PerfilQuadTemp}) such that they minimize the least-squares error (\ref{ErrorCuadraticoGenerico}) subject  to the conditions (\ref{FrontFija:1F-pos-tempinfty-A}), (\ref{TempFront:1F-pos-tempinfty-A}), (\ref{CondStefan:1F-pos-tempinfty-A}) and (\ref{FrontInicial:1F-pos-tempinfty-A}).

\begin{teo}\label{TeoP4}
If a free boundary $s_{_4}$ and a temperature $T_{_4}$ constitute  a solution to problem (P$_{_4}$) then they are given by the expressions:
\begin{eqnarray}
T_{_4}(x,t)&=&t^{\alpha/2}\theta_{_\infty} \left[ A_{_4}\left(1-\frac{x}{s_{_{4}}(t)}\right) +B_{_{4}} \left(1-\frac{x}{s_{_{4}}(t)}\right)^{2}\right],\label{PerfilT4} \\
s_{_{4}}(t)&=& 2a\nu_{_{4}} \sqrt{t},
\end{eqnarray}
where the constants $A_{_{4}}$ and $B_{_{4}}$ are defined as a function of $\nu_{_{4}}$ as
\begin{align} \label{A4}
A_{_4}=\frac{2^{\alpha+1}  \nu_{_4}^{\alpha+2}  }{\mathrm{Ste}}, \qquad \qquad B_4= 1- \frac{2^{\alpha+1}  \nu_{_4}^{\alpha+2}  }{\mathrm{Ste}},
\end{align}
and where $\nu_{_{4}}>0$ must minimize for every $t>0$, the function
\begin{align}
E(\xi)=\frac{t^{\alpha-2} \theta_{_\infty}^2}{60 \mathrm{Ste}^2} \frac{p(\xi)}{\xi^4},\quad \forall t>0 \label{EcNu4}
\end{align}
with 
\begin{align}
p(\xi)&= \xi^{8+2\alpha} 2^{2\alpha+1}(\alpha^2+\alpha+4)+ 5 \;\xi^{2\alpha+6} 2^{2\alpha+2}(1+\alpha)+ 15\;\xi^{2\alpha+4}  2^{2\alpha+2}\nonumber\\
&+\xi^{\alpha+6} 2^{\alpha} \mathrm{Ste} (2+3\alpha+3\alpha^2)+5\; \xi^{\alpha+4} 2^{\alpha+1}\mathrm{Ste} (1+\alpha)\nonumber\\
&-15\; \xi^{2+\alpha} 2^{\alpha+2} 2^{\alpha+2} \mathrm{Ste}+\xi^4 \mathrm{Ste}^2 (2+3\alpha+3\alpha^2)\nonumber\\
&-10\xi^2\mathrm{Ste}^2 (1+\alpha)+15\mathrm{Ste}^2.\label{p4}
\end{align}
\end{teo} 

\begin{proof}
Provided that $T_{_4}$ adopts a quadratic profile in space given by (\ref{PerfilT4}), then the condition (\ref{TempFront:1F-pos-tempinfty-A}) holds immediately and the Stefan condition  (\ref{CondStefan:1F-pos-tempinfty-A}) becomes equivalent to
\be
-kt^{\alpha/2}\theta_{_\infty}\frac{A_{_4}}{s_{_4}(t)}=-\gamma s_{_4}^{\alpha}(t)\dot s_{_4}(t).
\ee
Then
\be
s_{_4}(t)=\left( \frac{(\alpha+2)}{(\frac{\alpha}{2}+1)} \frac{k\theta_{_\infty}}{\gamma}A_{_4}\right)^{1/(\alpha+2)} \sqrt{t}.
\ee
Introducing $\nu_{_4}=\frac{1}{2a}\left( \frac{(\alpha+2)}{(\frac{\alpha}{2}+1)} \frac{k\theta_{_\infty}}{\gamma}A_{_4}\right)^{1/(\alpha+2)}$, the free boundary becomes
\be 
s_{_4}(t)=2a\; \nu_{_4}\sqrt{t},\label{s4}
\ee
and
\be\label{RelA4}
A_{_4}=\frac{2^{\alpha+1} \nu_{_4}^{\alpha+2}}{\Ste} .
\ee
In addition, from the boundary condition at the fixed face (\ref{FrontFija:1F-pos-tempinfty-A}) we get
\be \label{RelA4B4-2}
A_{_4}+B_{_4}=1.
\ee
Then we obtain formulas (\ref{A4}) for the coefficients $A_{_4}$ and $B_{_4}$.
Finally, as the free boundary $s_{_4}$ is defined by (\ref{s4}), we have to minimize the least-squares error 
$E$ given by \ref{ErrorCuadraticoGenerico2}. In addition,  replacing $A_{_4}$ and $B_{_4}$ by  the formulas given in (\ref{A4})  we get that $\nu_{_4}$ must minimize (\ref{EcNu4}).

\end{proof}

\begin{corollary} \label{CorolarioAlpha0-Temp}For the classical Stefan problem, i.e. for the case $\alpha=0$, we get that problem (P$_{_4}$) has a unique solution given by 
\begin{eqnarray}
T_{_{4}}^{(0)}(x,t)&=&\theta_{_\infty} \left[ A_{_4}^{(0)}\left(1-\frac{x}{s_{_{4}^{(0)}}(t)}\right) +B_{_{4}^{(0)}} \left(1-\frac{x}{s_{_{4}^{(0)}}(t)}\right)^{2}\right],\label{PerfilT40} \\
s_{_{4}}^{(0)}(t)&=& 2a\nu_{_{4}}^{(0)} \sqrt{t},
\end{eqnarray}
where the superscript $(0)$ makes reference to the value of $\alpha=0$ and the constants $A_{_{4}}^{(0)}$ and $B_{_{4}}^{(0)}$ are defined as a function of $\nu_{_{4}}^{(0)}$ as
\begin{align} \label{A40}
A_{_4}^{(0)}=\frac{2 (\nu_{_4}^{(0)})^{2}  }{\mathrm{Ste}}, \qquad \qquad B_4= 1- \frac{2 (\nu_{_4}^{(0)})^{2}  }{\mathrm{Ste}}
\end{align}
being $\nu_{_{4}}^{(0)}>0$  the value where the function $E^{(0)}$ attains its minimum
\begin{align}
E^{(0)}(\xi)=\frac{t^{-2} \theta_{_\infty}^2}{60 \mathrm{Ste}^2} \frac{p^{(0)}(\xi)}{\xi^4},\quad \forall t>0 \label{EcNu40}
\end{align}
with 
\begin{align}
p^{(0)}(\xi)&=8\xi^8+2(10+\mathrm{Ste})\xi^6+2(30+5\mathrm{Ste}+\mathrm{Ste}^2)\xi^4\nonumber\\
&-10\mathrm{Ste}(6+\mathrm{Ste})\xi^2+15\mathrm{Ste}^2 \label{p40}
\end{align}
In addition, $\nu_{_4}^{(0)}$ can be obtained as the unique positive root of the following real polynomial
\begin{equation}
r(\xi)=32\xi^8+4(10+\mathrm{Ste})\xi^6+20\mathrm{Ste}(6+\mathrm{Ste})\xi^2-60\mathrm{Ste}^2.\label{r}
\end{equation}

\end{corollary}

\begin{remark} Due to formula (\ref{EcNu40}) we have that the error we commit when approximating with problem (P$_{_4}$) for the case $\alpha=0$ is inversely proportional to the square of time, i.e $E^{(0)}\propto 1/t^2$.
\end{remark}

\begin{proof}
From Theorem \ref{TeoP4} we need to minimize the function $E(\xi)$ given by (\ref{EcNu4}) for the case $\alpha=0$. 
So, it is clear evident that we need to minimize the function $E^{(0)}(\xi)$ given by (\ref{EcNu40}) which is equivalent to minimize the function $F^{(0)}(\xi)=\frac{p^{(0)}(\xi)}{\xi^4}$. Therefore, let us show that $F^{(0)}$ has a unique positive value where the minimum is attained. Observe that $F^{(0)}$ is a continuous function in $\mathbb{R}^{+}$. Moreover if we compute its derivative we obtain
$$F'^{(0)}(\xi)=\frac{r(\xi)}{\xi^5}$$
with $r$ given by (\ref{r}).
As  $r$ is a polynomial that verifies $r(0)=-60 \Ste^2<0$, $r(+\infty)=+\infty$, and $r'(\xi)>0$, $\forall \xi>0$, we obtain that  there exists a unique value $\xi_0>0$ such that $r(\xi_0)=0$. In addition, we can assure that $r(\xi)<0$,  for every $\xi<\xi_0$ and $r(\xi)>0$, for every $\xi>\xi_0$. Consequently we have
$$F'^{(0)}(\xi)<0, \;\;\forall \xi<\xi_0,\qquad F'^{(0)}(\xi_0)=0,\qquad F'^{(0)}(\xi)>0, \;\; \forall \xi>\xi_0. $$
We can conclude that $F^{(0)}$ decreases in $(0,\xi_{0})$ and increases in $(\xi_0,+\infty)$. This means that $F^{(0)}$ has a unique minimum that is attained at $\xi_0$. Calling $\nu_{_4}^{(0)}=\xi_0$  we get that $\nu_{_4}^{(0)}$ is the unique positive root of $r$ and minimizes the error function $E^{(0)}$.
\end{proof}

Taking into account  the last result we show in the Table \ref{Tabla:NuiVsNu4}  the coefficient $\nu$ that characterizes the exact free boundary of problem (P), the approximate coefficient $\nu_{_2}$ obtained by the modified integral balance method (which until now was the most accurate technique) and the coefficient $\nu_{_4}$ defined by the Corollary \ref{CorolarioAlpha0-Temp} for different values of $\Ste$ numbers. Computing also the percentage relative error committed in each case we assure that the approximate problem (P$_{_4}$) is the best approximation we can obtain adopting a quadratic profile in space for the temperature.

 \begin{table}
\small
\caption{{\footnotesize Dimensionless coefficients of the free boundaries and their percentage relative error for $\alpha=0$.}}
\label{Tabla:NuiVsNu4}  
\begin{center}
\begin{tabular}{cc|cc|cc}
\hline
Ste      & $\nu$             &   $ \nu_{_2} $  & $E_{\text{rel}}(\nu_{_2})$      &   $\nu_{_4}$ & $E_{\text{rel}}(\nu_{_4})$   \\
\hline
	0.1  &  0.2200     &   0.2209 &   0.3947 \%  &   0.2209  &  0.3855 \% \\
    0.2  &  0.3064     &   0.3087 &   0.7499 \%  &   0.3086  &  0.7168 \% \\
    0.3  &  0.3699     &   0.3738 &   1.0707 \%  &   0.3736  &  1.0040 \% \\
    0.4  &  0.4212     &   0.4270 &   1.3618 \%  &   0.4265  &  1.2551 \% \\
    0.5  &  0.4648     &   0.4723 &   1.6266 \%  &   0.4716  &  1.4762 \% \\
    0.6  &  0.5028     &   0.5122 &   1.8683 \%  &   0.5112  &  1.6722 \% \\
    0.7  &  0.5365     &   0.5477 &   2.0895 \%  &   0.5464  &  1.8470 \% \\
    0.8  &  0.5669     &   0.5799 &   2.2923 \%  &   0.5783  &  2.0037 \% \\
    0.9  &  0.5946     &   0.6094 &   2.4786 \%  &   0.6074  &  2.1449 \% \\
    1.0  &  0.6201     &   0.6365 &   2.6500 \%  &   0.6342  &  2.2727 \% \\
    \hline
\end{tabular}
\end{center}
\end{table}

 \newpage
 
 In a similar way, we can define a new approximate \textbf{ problem (P$_{_{4h}}$)} for the problem (P$_{h}$) that consists in finding the free boundary $s_{_{4h}}=s_{_{4h}}(t)$ and the temperature $T_{_{4h}}=T_{_{4h}}(x,t)$ in $0<x<s_{_{4h}}(t)$ given by the profile (\ref{PerfilQuadTemp}) such that they minimize the least-squares error (\ref{ErrorCuadraticoGenerico}) subject to to the conditions (\ref{FrontFija:1F-pos-tempinfty-A}$^\star$), (\ref{TempFront:1F-pos-tempinfty-A})-(\ref{FrontInicial:1F-pos-tempinfty-A}).

\begin{teo} \label{TeoP4h}
If a free boundary $s_{_{4h}}$ and a temperature $T_{_{4h}}$ constitute  a solution to problem (P$_{_{4h}}$) then they are given by the expressions:
\begin{eqnarray}
T_{_{4h}}(x,t)&=&t^{\alpha/2}\theta_{_\infty} \left[ A_{_{4h}}\left(1-\frac{x}{s_{_{4h}}(t)}\right) +B_{_{4h}} \left(1-\frac{x}{s_{_{4h}}(t)}\right)^{2}\right],\label{PerfilT4h} \\
s_{_{4h}}(t)&=& 2a\nu_{_{4h}} \sqrt{t},\label{s4h}
\end{eqnarray}
where the constants $A_{_{4h}}$ and $B_{_{4h}}$ are defined as a function of $\nu_{_{4h}}$ as
\begin{align} \label{A4h}
A_{_{4h}}=\frac{2^{\alpha+1}  \nu_{_{4h}}^{\alpha+2}  }{\mathrm{Ste}}, \qquad \qquad B_{_{4h}}= \frac{2 \mathrm{Bi}\; \nu_{_{4h}}-A_{_{4h}} (1+2\mathrm{Bi}\; \nu_{_{4h}})}{2(1+\mathrm{Bi}\; \nu_{_{4h}})}
\end{align}
and where $\nu_{_{4h}}>0$ must minimize for every $t>0$, the real function:
\begin{align}
E_{_h}(\xi)&=\frac{t^{\alpha-2} \theta_{_\infty}^2}{60\; \mathrm{Ste}^2  (\frac{1}{\mathrm{Bi}}+\xi)^2}\; \cdot \; \nonumber\\
&  \left\lbrace p(\xi)+ \frac{1}{{\mathrm{Bi}}} \left[ 2^{2 \alpha } \left(7 \alpha ^2+7 \alpha +18\right) x^{2 \alpha +7}  +25\ 2^{2 \alpha +1} (\alpha +1) x^{2 \alpha +5},\right.\right.\nonumber\\
& +2^{\alpha } \left(9 \alpha ^2+9 \alpha +6\right) \mathrm{Ste} x^{\alpha +5}+15\ 2^{2 \alpha +2} x^{2 \alpha +3}\nonumber\\
&\left.-5\; 2^{\alpha +1} (\alpha +1) \mathrm{Ste}\;  x^{\alpha +3} -15 \;2^{\alpha +1} \mathrm{Ste}\; x^{\alpha +1}\right]\nonumber\\
&+ \left. \frac{1}{\mathrm{Bi}^2} \left[4^{\alpha +1} \left(2 \alpha ^2+2 \alpha +3\right) x^{2 \alpha +6}+ 5\ 4^{\alpha +1} (\alpha +1) x^{2 \alpha +4}+ 15\ 4^{\alpha } x^{2 \alpha +2} \right] \right\rbrace  \label{EcNu4h}
\end{align}
with $p(\xi)$ given by formula (\ref{p4}) 

\end{teo}

 \begin{proof}
It is clear immediate that the chosen profile temperature leads the condition (\ref{TempFront:1F-pos-tempinfty-A}) to be automatically verified.
From condition (\ref{CondStefan:1F-pos-tempinfty-A})  we obtain
\be
-kt^{\alpha/2}\theta_{_\infty}\frac{A_{_{4h}}}{s_{_{4h}}(t)}=-\gamma s_{_{4h}}^{\alpha}(t)\dot s_{_{4h}}(t).
\ee
Therefore
\be
s_{_{4h}}(t)=\left( \frac{(\alpha+2)}{(\frac{\alpha}{2}+1)} \frac{k\theta_{_\infty}}{\gamma}A_{_{4h}}\right)^{1/(\alpha+2)} \sqrt{t}.
\ee
Introducing the new coefficient $\nu_{_{4h}}$ such that $ \nu_{_{4h}}=\frac{1}{2a}\left( \frac{(\alpha+2)}{(\frac{\alpha}{2}+1)} \frac{k\theta_{_\infty}}{\gamma}A_{_{4h}}\right)^{1/(\alpha+2)}$, the free boundary can be expressed as
\be 
s_{_{4h}}(t)=2a\; \nu_{_{4h}}\sqrt{t},
\ee
where the following equality holds
\be\label{Ec1:P4h}
A_{_{4h}}=\frac{2^{\alpha+1} \nu_{_{4h}}^{\alpha+2}}{\Ste} .
\ee
The  convective boundary condition at $x=0$, i.e. condition (\ref{FrontFijaConvectiva}), leads to
\be \label{Ec2:P4h}
A_{_{4h}}(1+2\Bt\; \nu_{_{4h}})+2B_{_{4h}}(1+\Bt\;\nu_{_{4h}})=2\Bt\;\nu_{_{4h}}.
\ee

Therefore we obtain the formulas  given in (\ref{A4h}). Replacing $A_{_{4h}}$, $B_{_{4h}}$ and $s_{_{4h}}$ for their expressions in function of $\nu_{_{4h}}$, minimizing the least-squares error (\ref{ErrorCuadraticoGenerico}) is equivalent to minimizing (\ref{EcNu4h}) (obtained by Mathematica software).
\end{proof}

\begin{corollary} \label{CorolarioAlpha0-Conv}For the classical Stefan problem, i.e. for the case $\alpha=0$, we get that if  $\mathrm{Bi}>\frac{1}{\sqrt{12}}$ and $\mathrm{Ste}<\frac{1}{2\mathrm{Bi}^2}$, then (P$_{_{4h}}$) has a unique solution given by 
\begin{eqnarray}
T_{_{4h}}^{(0)}(x,t)&=&\theta_{_\infty} \left[ A_{_{4h}}^{(0)}\left(1-\frac{x}{s_{_{4h}^{(0)}}(t)}\right) +B_{_{4h}^{(0)}} \left(1-\frac{x}{s_{_{4h}^{(0)}}(t)}\right)^{2}\right],\label{PerfilT40h} \\
s_{_{4h}}^{(0)}(t)&=& 2a\nu_{_{4h}}^{(0)} \sqrt{t},
\end{eqnarray}
where the superscript $(0)$ makes reference to the value of $\alpha=0$ and the constants $A_{_{4h}}^{(0)}$ and $B_{_{4h}}^{(0)}$ are defined as a function of $\nu_{_{4h}}^{(0)}$ as
\begin{align} \label{A40h}
A_{_{4h}}^{(0)}=\frac{2 (\nu_{_{4h}}^{(0)})^{2}  }{\mathrm{Ste}}, \qquad \qquad B_{_{4h}}^{(0)}=\frac{2\mathrm{Bi}\;\nu_{_{4h}}^{(0)}-A_{_{4h}}^{(0)} (1+2\mathrm{Bi}\nu_{_{4h}}^{(0)} ) }{2(1+\nu_{_{4h}}^{(0)} \mathrm{Bi})}
\end{align}
being $\nu_{_{4h}}^{(0)}>0$  the value where the function $E_{h}^{(0)}$ attains its minimum
\begin{align}
E_{h}^{(0)}(\xi)&=\frac{t^{-2} \theta_{_\infty}^2}{60 \mathrm{Ste}^2 x^2(\frac{1}{\mathrm{Bi}}+x)^2}\left\lbrace p^{(0)}(\xi)+\frac{1}{{\mathrm{Bi}}}\left[2x(9x^6+(3\mathrm{Ste}+25)x^4\right.\right.\nonumber\\
&\left.\left. +5(6-\mathrm{Ste})x^2-15\mathrm{Ste}  )\right]+ \frac{1}{{\mathrm{Bi}^2}} x^2(12x^4+20x^2+15)\right\rbrace \label{EcNu40h}
\end{align}
where $p^{(0)}$ is given by (\ref{p40}).
Moreover, $\nu_{_{4h}}^{(0)}$ can be obtained as the unique positive root of the following polynomial:
\begin{align}
r_{_h}(\xi)&= 16 \mathrm{Bi}^3 \xi^9+51 \mathrm{Bi}^2 \xi^8+\xi^7 \left(2 \mathrm{Bi}^3 \mathrm{Ste}+20 \mathrm{Bi}^3+57 \mathrm{Bi}\right)\nonumber\\
&+\xi^6 \left(7 \mathrm{Bi}^2 \mathrm{Ste}+65 \mathrm{Bi}^2+24\right)\nonumber\\
&+\mathrm{Bi} (3 \mathrm{Ste}+25) \xi^5+\xi^4 \left(\mathrm{Bi}^2 \left(2 \mathrm{Ste}^2+15 \mathrm{Ste}+30\right)+20\right)\nonumber\\
&+5 \mathrm{Bi} (3 + (-1 + 12 \mathrm{Bi}^2)  \mathrm{Ste} + 2 \mathrm{Bi}^2  \mathrm{Ste}^2) \xi^3+45 \mathrm{Bi}^2 \mathrm{Ste} \xi^2\nonumber\\
&+15 \mathrm{Bi}  \mathrm{Ste} (1 - 2 \mathrm{Bi}^2 \mathrm{Ste}) \xi-15 \mathrm{Bi}^2  \mathrm{Ste}^2.\label{rh}
\end{align}
\end{corollary}

\begin{proof}
When we replace $\alpha=0$ in Theorem \ref{TeoP4h} we immediately obtain the formulas (\ref{A40h}) and (\ref{EcNu40h}).
In order to prove that there exists a unique value that minimizes the least-squares error, we compute $E'_{h}(\xi)$ and we get that $E_{h}'(\xi)=\frac{\theta_{_\infty}}{30 \mathrm{Ste}^2 t^2 \xi^3 (\mathrm{Bi} \xi+1)^3}r_{h}(\xi)$ with $r_{h}$ given by (\ref{rh}). We can observe that $r_{h}(0)<0$, $r_{h}(+\infty)=+\infty$, $r'_{h}>0$ under the hypothesis that  $\mathrm{Bi}>\frac{1}{\sqrt{12}}$ and $\mathrm{Ste}<\frac{1}{2\mathrm{Bi}^2}$. Therefore, we can assure that there exists a unique $\xi_{_{h0}}$ such that $r_{h}(\xi_{_{h0}})=0$. In addition we have that $r_{h}(\xi)<0$, $\forall \xi<\xi_{_{h0}}$ and $r_{h}(\xi)>0$, $\forall \xi>\xi_{_{h0}}$. Then we get that $E_{h}(\xi)$ decreases $\forall \xi<\xi_{_{h0}}$ and increases $\forall \xi>\xi_{_{h0}}$. Consequently we obtain that $\xi_{_{h0}}$ constitutes the unique minimum of the least-squares error.

\end{proof}

In view of the above result, for $\alpha=0$ we compare the coefficient $\nu_{h}$ that characterizes the exact free boundary problem with the coefficient $\nu_{_{2h}}$ corresponding to the modified integral method, which was until now the most accurate, and we also compare with the coefficient $\nu_{_{2h}}$ obtained when minimizing the least-squares error. We fix $\Ste=0.02$ and vary $\Bt$ between 1 and 5. The value of this parameters are chosen in order to verify the hypothesis of Corollary \ref{CorolarioAlpha0-Conv}. By computing the percentage relative error of each method we conclude that the approximate problem (P$_{_{4h}}$) gives us the best approximate solution to problem (P$_{h}$).

 \begin{table}
\small
\caption{{\footnotesize Coefficients of the free boundaries and their percentage relative error for $\alpha=0$ and $\Ste=0.02$.}}
\label{Tabla:NuihVsNu4h}  
\begin{center}
\begin{tabular}{cc|cc|cc}
\hline
$\Bt$      & $\nu_{h}$             &   $ \nu_{_{2h}} $  & $E_{\text{rel}}(\nu_{_{2h}})$      &   $\nu_{_{4h}}$ & $E_{\text{rel}}(\nu_{_{4h}})$   \\
\hline
    1.0000  &  0.0193 &   0.0193  &  0.0002 \% &   0.0193 &   0.0002 \% \\
    2.0000  &  0.0350 &   0.0350  &  0.0023 \% &   0.0350 &   0.0022 \%\\
    3.0000  &  0.0468 &   0.0468  &  0.0066 \% &   0.0468 &   0.0064 \%\\
    4.0000  &  0.0553 &   0.0553  &  0.0120 \% &   0.0553 &   0.0117 \%\\
    5.0000  &  0.0617 &   0.0617  &  0.0175 \% &   0.0617 &   0.0172 \%\\
     \hline
\end{tabular}
\end{center}
\end{table}
    
    In case we decide to use the formula  (\ref{rh}) to compute $\nu_{_{4h}}$ without satisfying the hypothesis of the Corollary \ref{CorolarioAlpha0-Conv}, fixing $\Ste=0.5$ and varying $\Bt$ from 1 to 100 we get the results shown int Table \ref{Tabla:NuihVsNu4h1}.
    
     \begin{table}
\small
\caption{{\footnotesize Coefficients of the free boundaries and their percentage relative error for $\alpha=0$ and $\Ste=0.02$.}}
\label{Tabla:NuihVsNu4h1}  
\begin{center}
\begin{tabular}{cc|cc|cc}
\hline
$\Bt$      & $\nu_{h}$             &   $ \nu_{_{2h}} $  & $E_{\text{rel}}(\nu_{_{2h}})$      &   $\nu_{_{4h}}$ & $E_{\text{rel}}(\nu_{_{4h}})$   \\
\hline
    1 &    0.2926 &   0.2937  &  0.3939 \% &   0.2933 &   0.2600 \%\\
   10 &    0.4422 &   0.4484  &  1.4111 \% &   0.4477 &   1.2478  \%\\
   20 &    0.4533 &   0.4602  &  1.5151 \% &   0.4595 &   1.3576 \%\\
   30 &    0.4571 &   0.4642  &  1.5514 \% &   0.4635 &   1.3962 \%\\
   40 &    0.4590 &   0.4662  &  1.5699 \% &   0.4655 &   1.4158 \%\\
   50 &    0.4601 &   0.4674  &  1.5811 \% &   0.4667 &   1.4277 \%\\
   60 &    0.4609 &   0.4682  &  1.5886 \% &   0.4675 &   1.4357 \%\\
   70 &    0.4615 &   0.4688  &  1.5940 \% &   0.4681 &   1.4414 \%\\
   80 &    0.4619 &   0.4693  &  1.5980 \% &   0.4686 &   1.4457 \%\\
   90 &    0.4622 &   0.4696  &  1.6012 \% &   0.4689 &   1.4491 \%\\
  100 &    0.4625 &   0.4699  &  1.6037 \% &   0.4692 &   1.4518 \%\\
\hline
    \end{tabular}
\end{center}
\end{table}
    \newpage

\section{Conclusions}

In this paper we have studied different approximate methods for one-dimensional one-phase Stefan problems where the main feature consists in taking a space-dependent latent heat. We have considered two different problems that differ from each other in their boundary condition at the fixed face $x=0$: Dirichlet or Robin condition. We have implemented the classical heat balance integral method, a modified integral method and the refined integral method. Exploiting the knowledge of the exact solution of both problems (available in literature), we have studied the accuracy of the different approximations obtained. All the analysis have been carried out using dimensionless parameters like Stefan number and Biot number. Furthermore we have studied the case when Bi goes to infinity in the problem with a convective condition, recovering the approximate  solutions  when a temperature condition is imposed at the fixed face.  We provided some numerical simulations and we have concluded that in the majority of  cases the modified integral method is the most reliable in terms of accuracy. When approaching by the minimization of the least-squares error, we get better approximations but only for the case $\alpha=0$ (where we could prove existence and uniqueness of solution). The least accurate method was the classical heat balance integral method, not only to the high percentage error committed but also because we could not obtain a result that assures uniqueness of the approximate solution.

\section*{Acknowledgement}

The present work has been partially sponsored by the Project PIP No. 0275 from CONICET-UA, Rosario, Argentina, by the Project ANPCyT PICTO Austral 2016 No. 0090, and by the European Union's Horizon 2020 Research and Innovation Programme under the Marie Sklodowska-Curie grant agreement 823731 CONMECH.

The authors would like to thank the two anonymous referees for their helpful comments.

\bibliographystyle{elsarticle-num} 
\bibliography{Biblio}

\begin{thebibliography}{10}
\expandafter\ifx\csname url\endcsname\relax
  \def\url#1{\texttt{#1}}\fi
\expandafter\ifx\csname urlprefix\endcsname\relax\def\urlprefix{URL }\fi
\expandafter\ifx\csname href\endcsname\relax
  \def\href#1#2{#2} \def\path#1{#1}\fi

\bibitem{AlSo1993}
V.~Alexiades, A.~D. Solomon, Mathematical modeling of melting and freezing
  processes, Hemisphere Publishing Corp., Washington, 1993.

\bibitem{Ca1984}
J.~R. Cannon, The one-dimensional heat equation, Addison-Wesley, Menlo Park,
  California, 1984.

\bibitem{Gu2003}
S.~C. Gupta, The classical {S}tefan problem. Basic concepts, modelling and
  analysis, Elsevier, Amsterdam, 2003.

\bibitem{Lu1991}
V.~J. Lunardini, Heat transfer with freezing and thawing, Elsevier Science
  Publishers B. V., 1991.

\bibitem{Ta2000}
D.~A. Tarzia, A bibliography on moving-free boundary problems for heat
  diffusion equation. {T}he {S}tefan problem, MAT-Serie A 2 (2000) 1--297.

\bibitem{Ta2011}
D.~A. Tarzia, Explicit and Approximated Solutions for Heat and Mass Transfer
  Problems with a Moving Interface, Prof. Mohamed El-Amin (Ed.), In Tech,
  Rijeka, 2011, Ch. 20, Advanced Topics in Mass Transfer, pp. 439--484.

\bibitem{SwVoPa2000}
J.~Swenson, V.~Voller, C.~Paola, G.~Parker, J.~Marr, Fluvio-deltaic
  sedimentation: a generalized stefan problem, European Journal of Applied
  Mathematics 11 (2000) 433--452.

\bibitem{ZhShZh2018}
Y.~Zhou, X.~Shi, G.~Zhou, Exact solution for a two-phase {S}tefan problem with
  power-type latent heat, Journal of Engineering Mathematics, $ $ 110 (2018)
  1--13.

\bibitem{RiMy2016}
H.~Ribera, T.~Myers, A mathematical model for nanoparticle melting with
  size‑dependent latent heat and melt temperature, Microfluid Nanofluid
  20-147 (2016) 1--13.

\bibitem{ZhBuLu2013}
Y.~Zhou, W.~Bu, M.~Lu, One-dimensional consolidation with a threshold gradient:
  a {S}tefan problem with rate-dependent latent heat, International Journal for
  Numerical and Analytical Methods in Geomechanics 37 (2013) 2825--2832.

\bibitem{VoSwPa2004}
V.~R. Voller, J.~Swenson, C.~Paola, An analytical solution for a {S}tefan
  problem with variable latent heat, International Journal of Heat and Mass
  transfer 47 (2004) 5387--5390.

\bibitem{ZhWaBu2014}
Y.~Zhou, Y.~Wang, W.~Bu, Exact solution for a {S}tefan problem with a latent
  heat a power function of position, International Journal of Heat and Mass
  Transfer 69 (2014) 451--454.

\bibitem{ZhXi2015}
Y.~Zhou, L.~Xia, Exact solution for {S}tefan problem with general power-type
  latent heat using {K}ummer function, International Journal of Heat and Mass
  Transfer 84 (2015) 114--118.

\bibitem{SaTa2011-a}
N.~N. Salva, D.~A. Tarzia, Explicit solution for a {S}tefan problem with
  variable latent heat and constant heat flux boundary conditions, Journal of
  Mathematical Analysis and Applications 379 (2011) 240--244.

\bibitem{BoTa2018-CAA}
J.~Bollati, D.~Tarzia, Explicit solution for the one-phase {S}tefan problem
  with latent heat depending on the position and a convective boundary
  condition at the fixed face, Communications in Applied Analysis 22 (2018)
  309--332.

\bibitem{BoTa2018-ZAMP}
J.~Bollati, D.~Tarzia, Exact solution for a two-phase {S}tefan problem with
  variable latent heat and a convective boundary condition at the fixed face,
  Z. Angew. Math. Phys. 69-38 (2018) 1--15.

\bibitem{Do2014}
T.~Dorian, Spatial and temporal variability of latent heating in the tropics
  using TRMM observations, Master of science thesis, University of
  Misconsin-Madison, 2014.

\bibitem{Mc1991}
T.~McConnell, The two-sided stefan problem with a spatially dependent latent
  heat, Transactions of the American Mathematical Society 326 (1991) 669--699.

\bibitem{ZhHuLiZhZh2018}
Y.~Zhou, X.~Hu, T.~Li, D.~Zhang, G.~Zhou, Similarity type of general solution
  for one-dimensional heat conduction in the cylindrical coordinate,
  International Journal of Heat and Mass Transfer 119 (2018) 542--550.

\bibitem{Pr1970}
M.~Primicerio, Stefan-like problems with space-dependent latent heat, Meccanica
  5 (1970) 187--190.

\bibitem{Go2002}
H.~Gottlieb, Exact solution of a {S}tefan problem in a nonhomogeneous cylinder,
  Applied Mathematics Letters 15 (2002) 167--172.

\bibitem{Go1958}
T.~Goodman, The heat balance integral methods and its application to problems
  involving a change of phase, Transactions of the ASME 80 (1958) 335--342.

\bibitem{Wo2001}
A.~S. Wood, A new look at the heat balance integral method, Applied
  Mathematical Modelling 25 (2001) 815--824.

\bibitem{SaSiCo2006}
N.~Sadoun, E.~Si-ahmed, J.~Colinet, On the refined integral method for the
  one-phase stefan problem with time-dependent boundary conditions, Applied
  Mathematical Modelling 30 (2006) 531--544.

\bibitem{Hr2009-a}
J.~Hristov, The heat-balance integral method by a parabolic profile with
  unspecified exponent:analysis and benchmark exercises, Thermal Science 13
  (2009) 27--48.

\bibitem{Hr2009-b}
J.~Hristov, Research note on a parabolic heat-balance integral method with
  unspecified exponent: An entropy generation approach in optimal profile
  determination, Thermal Science 13 (2009) 49--59.

\bibitem{Mi2012}
S.~L. Mitchell, Applying the combined integral method to one-dimensional
  ablation, Applied Mathematical Modelling 36 (2012) 127--138.

\bibitem{MiMy2010-a}
S.~L. Mitchell, T.~Myers, Improving the accuracy of heat balance integral
  methods applied to thermal problems with time dependent boundary conditions,
  International Journal of Heat and Mass Transfer 53 (2010) 3540--3551.

\bibitem{Ta2017}
D.~A. Tarzia, Relationship between {N}eumann solutions for two-phase
  {L}am\'e-{C}lapeyron-{S}tefan problems with convective and temperature
  boundary conditions, Thermal Science 21-1 (2017) 187--197.

\bibitem{BoSeTa2018}
J.~Bollati, J.~Semitiel, D.~Tarzia, Heat balance integral methods applied to
  the one-phase {S}tefan problem with a convective boundary condition at the
  fixed face, Applied Mathematics and Computation 331 (2018) 1--19.

\bibitem{RiMyMc2019}
H.~Ribera, T.~Myers, M.~Mac~Davette, Optimising the heat balance integral
  method in spherical and cilyndrical stefan problems, Applied Mathematics and
  Computation 354 (2019) 216--231.

\bibitem{BoTa2018-EJDE}
J.~Bollati, D.~Tarzia, One-phase {S}tefan problem with a latent heat depending
  on the position of the free boundary and its rate of change, Electronic
  Journal of Differential Equations 2018-10 (2018) 1--12.

\bibitem{So1979}
A.~D. Solomon, An easily computable solution to a two-phase {S}tefan problem,
  Solar energy 33 (1979) 525--528.

\end{thebibliography}

\end{document}